\documentclass[a4]{amsart}
\usepackage{fullpage}
\usepackage{nicefrac}
\usepackage{enumerate}
\usepackage{amssymb,mathrsfs,mathtools,amsthm,amscd}
\usepackage[english]{babel}
\usepackage[utf8]{inputenc}
\usepackage[T1]{fontenc}

\usepackage{comment}
\usepackage[normalem]{ulem}

\usepackage{etex,tikz}\usetikzlibrary{patterns}

\usepackage{xcolor}\usepackage[all]{xy}
\def\R{\mathbb{R}}
\def\C{\mathbb{C}} 
\def\N{\mathbb{N}}

\def\B{\mathbb{B}} 

\def\OO{\mathscr{O}}

\def\ep{\varepsilon}

\let\leq\leqslant \let\geq\geqslant

\DeclareMathOperator{\Supp}{Supp} 
\DeclareMathOperator{\rk}{rank} 
\DeclareMathOperator{\Ric}{Ric} 

\def\bZ{\mathbb{Z}}
\def\bQ{\mathbb{Q}}
\def\bR{\mathbb{R}}
\def\bC{\mathbb{C}}
\def\bV{\mathbb{V}}
\def\bF{\mathbb{F}}
\def\bH{\mathbb{H}}

\def\cD{\mathcal{D}}
\def\cE{\mathcal{E}}
\def\cF{\mathcal{F}}
\def\cH{\mathcal{H}}
\def\cL{\mathcal{L}}
\def\cO{\mathcal{O}}
\def\cT{\mathcal{T}}
\def\cV{\mathcal{V}}

\def\Hom{\mathrm{Hom}}
\def\End{\mathrm{End}}
\newcommand{\sbt}{\,\begin{picture}(-1,1)(-1,-3)\circle*{3}\end{picture}\ }
\def\rank{\mathrm{rank}}
\def\tr{\mathrm{tr}}

\def\displaymap#1#2{%
  \ifx\relax#1\relax
  \left.\vcenter{\xymatrix@=0pc{#2}}\right.
  \else
  #1\colon\left\vert\vcenter{\xymatrix@=0pc{#2}}\right.
  \fi
}

\newtheorem{theorem}{Theorem}[section]
\newtheorem{lemma}[theorem]{Lemma}
\newtheorem{proposition}[theorem]{Proposition}
\newtheorem{proposition?}[theorem]{??Proposition??}
\newtheorem{corollary}[theorem]{Corollary}

\newtheorem{example}{\it Example\/}

\theoremstyle{definition}
\newtheorem{definition}[theorem]{Definition}
\newtheorem{remark}[theorem]{\it Remark\/}

\author[D. Brotbek]{Damian Brotbek}
\address{\noindent D. Brotbek: Institut Élie Cartan, Université de Lorraine, Vand\oe{}uvre les Nancy, France}
\email{damian.brotbek@univ-lorraine.fr}

\author[Y. Brunebarbe]{Yohan Brunebarbe}
\address{\noindent Y. Brunebarbe:  IMB-CNRS, Université de Bordeaux, Talence, France.}
\email{yohan.brunebarbe@math.u-bordeaux.fr}

\title{Arakelov-Nevanlinna inequalities for variations of Hodge structures and applications}
\begin{document}

\begin{abstract}
We prove a Second Main Theorem type inequality for any log-smooth projective pair \((X,D)\) such that \(X\setminus D\) supports a complex polarized variation of Hodge structures. This can be viewed as a Nevanlinna theoretic analogue of the Arakelov inequalities for variations of Hodge structures due to Deligne, Peters and Jost-Zuo. As an application, we obtain in this context a criterion of hyperbolicity that we use to derive a vast generalization of a well-known hyperbolicity result of Nadel.

The first ingredient of our proof is a Second Main Theorem type inequality for any log-smooth projective pair \((X,D)\) such that \(X\setminus D\) supports a metric whose holomorphic sectional curvature is bounded from above by a negative constant. The second ingredient of our proof is an explicit bound on the holomorphic sectional curvature of the Griffiths-Schmid metric constructed from a variation of Hodge structures.

As a byproduct of our approach, we also establish a Second Main Theorem type inequality for pairs \((X,D)\) such that \(X\setminus D\) is hyperbolically embedded in \(X\).

\end{abstract}
\maketitle


\section{Introduction}

Consider a smooth projective complex algebraic variety $X$ and a variation of Hodge structures $\mathbb{V} = (\cL, \cF^{\sbt}, h)$ defined on the complementary of a divisor $D \subset X$, see Section \ref{Recollection in Hodge theory} for a reminder in Hodge theory. Given a non-compact Riemann surface $B$ and a non-constant holomorphic map $f: B \rightarrow X$ such that \(f(B) \not\subset D\), we would like to investigate how much the position of the image $f(B)$ in $X$ is constrained due to the existence of $\bV$. To address this problem we establish a close analog of what is traditionally called a \textit{Second Main Theorem} in Nevanlinna theory. As a corollary, in case $B$ is algebraic, we obtain a criterion implying that the map $f$ is algebraic.\\

In this paper, we will always assume that the Riemann surface $B$ is parabolic in the function-theoretic sense, meaning that any bounded subharmonic function defined on $B$ is constant. For example, any algebraic Riemann surface satisfies this property. It is known that a \emph{non-compact} Riemann surface $B$ is parabolic if and only if it admits a parabolic exhaustion function, i.e. a continuous proper function $\sigma : B \rightarrow [0, + \infty)$ such that $ \log \sigma$ is harmonic outside a compact subset of $B$ (see for example \cite[Theorem 10.12]{Sto77}).  For instance \(\C\) is parabolic, and in this case one can take \(\sigma(z)=|z|\). More generally, for any affine  algebraic Riemann surface $B$, any proper finite holomorphic map $\pi : B \rightarrow \bC$ yields the parabolic exhaustion function $z \mapsto |\pi(z)|$. Parabolic Riemann surfaces provide a suitable framework for value distribution theory (see for instance \cite{SNM66}, \cite{Wu70}, \cite{Griffiths-King}, \cite{Sto77}, \cite{Sto83} or \cite{P-S14}). We briefly recall the main notations here and refer to Section \ref{ssec:Nevanlinna} for more details.\\

Let $B$ be a non-compact parabolic Riemann surface equipped with a parabolic exhaustion function $\sigma$. One defines the \emph{weighted Euler characteristic} of the pair \((B,\sigma)\) to be the function
\[\mathfrak{X}_{B,\sigma}(r):=\int_1^r \chi(B(t)) \frac{dt}{t},\]
where \(B(r):=\{p\in B\ ;\ \sigma(p)<r\} \) and \(\chi(B(t))\) is the Euler characteristic of \(B(t)\). If \(\alpha\) is a current of type \((1,1)\) on \(B\), one defines its \emph{Nevanlinna characteristic function} by
\[T_\alpha(r)=\int_1^r\left(\int_{B(t)}\alpha\right)\frac{dt}{t}.\]
If \( X\) is a smooth  proper variety, \(L\) a line bundle on \(X\), \(D\) a divisor on \( X\) and \(f:B\to X\) a non-constant holomorphic map such that \(f(B)\not\subset D\), then one defines the \emph{characteristic function of \(f\) with respect to \(L\)} and the \emph{counting function of \(f\) with respect to \(D\)} to be respectively  the functions
\[T_{f,L}(r):=T_{f^*C_1(L,h)}(r)\quad \text{and}\quad N^{[1]}_{f,D}(r):=T_{[f^{-1}(D)]}(r).\]
Here \([f^{-1}(D)]\) denotes the integration current with respect to the discrete subset \(f^{-1}(D) \subset B\) and $C_1(L,h)$ is the first Chern form of $L$ with respect to a \(\mathscr{C}^{\infty}\) metric  \(h\) on \(L\) (the notation is somewhat abusive since the function \(T_{f,L}(r)\) depends on the choice of the metric $h$, but only up to a bounded function, see Section \ref{ssec:Nevanlinna}). The counting function describes the asymptotic behaviour of the (possibly infinite) intersection of \(f(B)\) and \(D\) without counting multiplicities. 

\subsection{An Arakelov-Nevanlinna inequality}
Let $X$ be a smooth projective complex algebraic variety and $\mathbb{V} = (\cL, \cF^{\sbt}, h)$ be a variation of complex polarized Hodge structures of length $w$ defined on the complementary of a normal crossing divisor $D \subset X$. Assume that $\cL$ has unipotent monodromies around the irreducible components of $D$. We denote by $\bar \cF^p$ the canonical Deligne-Schmid extension of $\cF^p$ to $X$ for any integer $p$ and by $\bar L_{\mathbb{V}} = \otimes_p \det \bar \cF^p$ the canonical extension of the Griffiths line bundle of $\mathbb{V}$.

The following theorem is the main result in this paper (see also Theorem \ref{thm:Arakelov-Nevanlinna inequality} for another Arakelov-Nevanlinna inequality).
\begin{theorem}\label{thm:SMT for VHS}
Let $X, D, \bV$ and $\bar L_{\mathbb{V}}$ as above. Let $B$ be a non-compact parabolic Riemann surface equipped with a parabolic exhaustion function $\sigma$ and a non-constant holomorphic map $f: B \rightarrow X$ such that \(f(B) \not\subset D\). Then, for any ample line bundle $A$ on $X$, there exists $C >0$ such that the inequality
 \[   T_{f, \bar L_{\mathbb{V}} }(r) \leqslant \frac{w^2 \cdot \rk \cL}{4} \cdot \left( - \mathfrak{X}_{B,\sigma}(r) + N^{[1]}_{f, D}(r) \right) + C \cdot \left(  \log r + \log T_{f, A}(r) \right)    \ \  \]
holds for all $r \in \bR_{\geq 1}$ outside a Borel subset of finite Lebesgue measure.
\end{theorem}

\begin{remark} In the situation of Theorem \ref{thm:SMT for VHS}, assume moreover that $B$ is affine algebraic and that the holomorphic map $f: B \rightarrow X$ is algebraic, or equivalently that $B$ is the complementary of finitely many (but at least one) points in a compact Riemann surface $\bar B$ and that $f$ extends to a holomorphic map \(\bar f: \bar B \to X\). Choose a finite algebraic map $\pi : B \rightarrow \bC$ and equip $B$ with the parabolic exhaustion function $\sigma : z \mapsto |\pi(z)| $. 
Then one easily verifies that the three functions $ T_{f, \,\bar L_{\mathbb{V}} }, \mathfrak{X}_{B,\sigma} $ and $N^{[1]}_{f, D}$ are equivalent to $\deg_{\bar B}(\bar f ^\ast \bar \cF^p) \cdot \log$,  $ \chi(B) \cdot \log$ and $ \deg((f^\ast(D))_{red}) \cdot \log$ respectively. Therefore the preceding theorem is implied in this case by the following Arakelov inequality \cite{peters2000arakelovtype, Jost-Zuo, Bruni-level}:
\[ \deg_{\bar B}(\bar f ^\ast \bar L_{\mathbb{V}} ) \leqslant \frac{w^2 \cdot \rk \cL}{4} \cdot \big( -\chi(B) + \deg((f^\ast(D))_{red}) \big) . \]
\end{remark}

\begin{remark}
Arguing as in \cite{Bruni-level}, one can prove that Theorem \ref{thm:SMT for VHS} holds more generally when the monodromy of $\cL$ at infinity is quasi-unipotent, with $ \bar L_{\mathbb{V}}$ denoting the Griffiths parabolic line bundle of the variation $\bV$, cf. \textit{loc. cit.}.
\end{remark}

\begin{remark}
One of our motivations is the following conjecture of Griffiths (see \cite[Conjecture 4.10.5]{N-W14} for an even more optimistic version): \emph{If \(X\) is a smooth projective variety, \(D\) a simple normal crossing divisor on \(X\) and \(A\) an ample line bundle on \(X\), then there exists a constant $\alpha >0$ such that for any Zariski dense entire curve \(f:\C\to X\) and every \(\ep>0\), the inequality}
\begin{equation}
T_{f,K_{X}(D)}(r)\leqslant \alpha \cdot N_{f,D}(r)+\ep \cdot  T_{f,A}(r).\label{eq:Griffiths}
\end{equation}
\emph{holds for all $r \in \bR_{\geq 1}$ outside a Borel subset of finite Lebesgue measure (the characteristic functions are defined with respect to the exhaustion function $z \mapsto |z|$).}

This conjecture has several far reaching consequences, for instance, it implies that if \(K_{X}(D)\) is big (i.e. the pair \((X,D)\) is of log-general type) then every entire curve in \(X\setminus D\) is algebraically degenerate. This is a version of the logarithmic Green-Griffiths-Lang conjecture.

It is therefore natural to try to establish the inequality \eqref{eq:Griffiths} or other similar inequalities, which are referred to as \emph{Second Main Theorems} in Nevanlinna Theory, when one already knows that the pair \((X,D)\) satisfies the Green-Griffiths-Lang conjecture. This is for instance the case when \(X\setminus D\) supports a complex polarized variation of Hodge structures whose period map is immersive at one point, cf. \cite{Brunebarbe-Cadorel}.

That conjecture of Griffiths has the following algebraic analogue: \emph{If \(X\) is a projective variety and \(D\) is a simple normal crossing divisor on \(X\) such that the pair \((X,D)\) is of general type, then there exists constants \(\alpha,\beta\in \R_+\) such that for any compact Riemann surface \(C\) and every non-constant algebraic map \(f:C\to X\) such that \(f(C)\not\subset D\), one has}
\begin{equation}
\deg_Cf^*K_{X}(D)\leqslant \alpha \cdot \left(\deg_C \left(f^*D \right)_{\rm red}+2g\left(C \right)-2 \right)+\beta.\label{eq:GriffithsAlg}
\end{equation}
Both Griffiths' conjecture and its algebraic counterpart are still widely open. For a more precise discussion on the analogies between  value distribution theory for parabolic Riemann surfaces and its algebraic counterpart, we refer to \cite{Gas09}.

From this point of view, Theorem \ref{thm:SMT for VHS} can be viewed as a \textit{Second Main Theorem} for varieties supporting a variation of Hodge structures, replacing the log-canonical bundle of the ambiant variety by the canonical bundle (in the sense of Griffiths) of the variation.
\end{remark}

\subsection{Examples}

\begin{example}
Endowing $\bC^2$ with a non-degenerate hermitian sesquilinear form $h$ of signature $(1,1)$, and identifying the upper half plane $\bH$ with the space of positive lines in $\bC^2$, we can view $\bH$ as the parameter space of complex Hodge  structures on $\bC^2$ with Hodge numbers $(1,1)$ and polarized by $h$, cf. Section \ref{Recollection in Hodge theory}. This induces a variation of complex polarized Hodge structures  on any quotient of $\bH$ by a torsion-free discrete subgroup of $SU(h) \simeq SU(1,1) \simeq SL(2, \bR)$.
This applies in particular to any smooth complex algebraic curve $X$ with $\chi(X)  <0$ (so that $X$ is the quotient of $\bH$ by a Fuchsian group)
and Theorem \ref{thm:SMT for VHS} specialized to that case recovers the Second Main Theorem from Nevanlinna Theory in one variable. 

\end{example}

\begin{example} 
The preceding example generalizes as follows. Let $\cD$ be a bounded symmetric domain. Gross \cite{Gross94} (in the tube domain case) and Sheng-Zuo \cite{Sheng-Zuo10} (in general) have constructed an $\mathrm{Aut}(\cD)$-equivariant variation of complex polarized Hodge structures  on $\cD$ of length $\rk (\cD)$ and whose smallest Hodge bundle can be identified with the automorphic line bundle of $\cD$. Moreover, both the canonical bundle of $\cD$ and the canonical bundle of $\bV$ are a positive tensor power of the automorphic line bundle of $\cD$. If $\Gamma$ is a net lattice in the group of holomorphic automorphisms of $\cD$, then the quotient $U := \Gamma \backslash \cD$ is (uniquely) a smooth complex algebraic variety and the variation $\bV$ descends to $U$. For any smooth compactification $X$ of $U$ such that $D = X - U$ is a normal crossing divisor, the canonical extension to $X$ of the Griffiths line bundle of $\mathbb{V}$ is proportional to $K_{X}(D)$.
Therefore, if we specialize Theorem \ref{thm:SMT for VHS} in that case, we recover the Second Main Theorem for compactified locally symmetric spaces due to Nadel \cite{Nad89} and to Aihara-Noguchi \cite{A-N91}.
\end{example}

\begin{example}
Theorem \ref{thm:SMT for VHS} can be applied to any smooth complex algebraic variety $X$ which is a fine moduli space for a class of varieties that satisfy an infinitesimal Torelli theorem, e.g. smooth projective curves or smooth projective varieties with a trivial canonical bundle (with a suitable level structure).
\end{example}

\subsection{Two Applications}

The following applications refine some recent results of Deng \cite{deng2020big}. However, our approach differs from his: on the one hand, while Deng uses an ad hoc Finsler metric, we rely on the Griffiths-Schmid metric, and on the other hand, while Deng relies on the big Picard theorem from \cite{deng2019picard}, we use our Theorem \ref{thm:SMT for VHS}.

\subsubsection{A criterion of hyperbolicity}
Following \cite{Ariyan-Robert} we say that a complex algebraic variety $Y$ is Borel hyperbolic modulo a closed subvariety $Z \subset Y$ if, for every (reduced) complex algebraic variety $S$, any holomorphic map $f : S^{an} \rightarrow Y^{an}$ such that $f (S^{an}) \not \subset Z^{an}$ is algebraic. Observe that this implies that $Y$ is Brody hyperbolic modulo $Z$, i.e. that every non-constant holomorphic map $f : \bC \rightarrow Y^{an}$ satisfies $f(\bC) \subset Z^{an}$, since otherwise both $f$ and $f \circ \exp$ would be algebraic.

\begin{theorem} \label{thm:criterion for Borel hyperbolicity}
Let $X$ be a smooth projective complex algebraic variety and let $\mathbb{V} = (\cL, \cF^{\sbt}, h)$ be a variation of complex polarized Hodge structures of length $w$ defined on the complementary of a normal crossing divisor $D \subset X$. Assume that $\cL$ has unipotent monodromies around $D$ and let $\bar L_{\mathbb{V}}$ be the canonical extension of the Griffiths line bundle of $\mathbb{V}$. If one denotes by $\B_{+}$ the augmented base locus of the $\bQ$-line bundle $ \bar L_{\mathbb{V}} (- (  \frac{w^2}{4} \cdot \rk \cL)  \cdot D)$, then $X$ is Borel hyperbolic modulo $D \cup \B_{+}$.
\end{theorem}

\begin{remark}
Recall that for any $\bQ$-line bundle $L$ on a projective complex variety $Y$, its augmented base locus $\B_{+}(L)$ is the intersection over all ample $\bQ$-divisor $A$ on $Y$
of the stable loci  $\B(L(-A)) := \cap_{n \geq1} \mathrm{Bs}(L(-A)^{\otimes n})$.
\end{remark}

\subsubsection{A generalization of a theorem of Nadel}
Let $X$ be a (non necessarily smooth nor proper) complex algebraic variety equipped with a $\bZ$-local system $\cL$. For every prime number $p$ we have an induced $\bF_p$-local system $\cL \otimes_{\bZ} \bF_p$ and we denote by $X(p)$ the finite \'{e}tale cover of $X$ that trivializes the local system of sets $\mathcal{I}som_{\bF_p}(\bF_p^{\rk \cL},\cL \otimes_{\bZ} \bF_p)$. If $X$ is connected, then after fixing a base point $x \in X$, the covering map $X(p) \rightarrow X$ corresponds to the action of $\pi_1(X,x)$ on the set of basis of the $\bF_p$-vector space $(\cL \otimes_{\bZ} \bF_p)_x$. Note that this action is not necessarily transitive, so that $X(p)$ might not be connected. 

The following result is a vast generalization of the main result of \cite{Nad89}.
\begin{theorem}\label{thm:generalization of Nadel}
Assume that $\cL$ underlies a variation of complex polarized Hodge structures with a quasi-finite period map. Then, for all but finitely many prime numbers $p$, any proper algebraic variety $\overline{X(p)}$ that compactifies $X(p)$ is Borel hyperbolic modulo the boundary $\overline{X(p)} \backslash X(p)$.
\end{theorem}

\subsection{A Second Main Theorem for pseudo-metrics with negative holomorphic sectional curvature}
The first main input of the proof of Theorem \ref{thm:SMT for VHS} is the following \textit{Second Main Theorem} for pseudo-metrics with negative holomorphic sectional curvature (see also Theorem \ref{thm:SMTlocal}). 

\begin{theorem}[{= Theorem \ref{thm:SMTglobal}}]\label{thm:SMT for pseudo-metrics with negative holomorphic sectional curvature}
Let \(X\) be a smooth projective variety and let \(D\) be a reduced  divisor on \(X\). Suppose that \(X \setminus D\) is endowed with a Finsler pseudo-metric $h$ of class \(\mathscr{C}^{\infty}\) with holomorphic sectional curvature bounded above by \(-\gamma<0\). Assume that the degeneracy set of $h$ is contained in a nowhere-dense closed analytic subset of $X \setminus D$. Let $B$ be a non-compact parabolic Riemann surface equipped with a parabolic exhaustion function $\sigma$ and a non-constant holomorphic map \(f:B\to X\) whose image $f(B)$ is not contained in $D$ nor in the degeneracy set of $h$. If  \(\omega\) denotes the \((1,1)\)-form associated to the induced pseudo-metric $f^\ast h$ on \(B\setminus (f^{-1} (D))_{\rm red}\), then \(\omega\) is locally integrable on \(B\) and there exists a positive real number $C$ such that the following inequality
\begin{equation*}
T_{[\omega]}(r)\leqslant \frac{2\pi}{\gamma}\big(N^{[1]}_{f,D}(r) -\mathfrak{X}_{\sigma}(r)\big)+C \cdot (\log r+\log T_{[\omega]}(r))
\end{equation*}
holds for all $r \in \R_{\geq 1}$ outside a subset of finite Lebesgue measure.
\end{theorem}

Our proof of this result is very much inspired by the work of Nadel \cite{Nad89} and Aihara-Noguchi \cite{A-N91}. We generalize their arguments to parabolic Riemann surfaces by using a slight modification of the approach of P\u{a}un and Sibony \cite{P-S14} to Nevanlinna Theory.

\subsection{The holomorphic sectional curvature of the Griffiths-Schmid metric}
The second main input towards the proof of Theorem \ref{thm:SMT for VHS} is the following result that describes a pseudo-metric with nice negative curvature properties on any manifold supporting a variation of Hodge structures. For the most part it is well-known to experts, but to our knowledge the upper bound on the holomorphic sectional curvature is new.

\begin{theorem}\label{thm:holomorphic sectional curvature for VHS}
Let $(\mathcal{L}, \cF^{\sbt}, h)$ be a variation of complex polarized Hodge structures of length $w$ on a complex manifold $S$. The first Chern form of its Griffiths line bundle $L_{\mathbb{V}} := \otimes_p \det \cF^p$ equipped with the hermitian metric induced by $h$ is a closed positive real $(1,1)$-form, which coincides after multiplication by $2 \pi$ with the pull-back of the K\"ahler form of the Griffiths-Schmid metric on the corresponding  period domain. The corresponding K\"ahler pseudo-metric is non-degenerate on the Zariski-open subset of $S$ where the period map is immersive, has non-positive holomorphic bisectional curvature and its holomorphic sectional curvature $- \gamma$ satisfies 
\[ \frac{1}{\gamma}   \leqslant \frac{w^2}{4} \cdot \rk( \cL). \] 
\end{theorem}

\begin{remark}
Note that when $(\mathcal{L}, \cF^{\sbt}, h)$ is isomorphic to its dual (this holds for example if it is a variation of real polarized Hodge structures), then 
\[ \frac{w^2}{4}  \cdot \rk( \cL) =   \frac{w}{2} \cdot \sum_{i= 1}^p \rk( \cF^p) . \]
\end{remark}

\subsection{Second Main Theorem for hyperbolically embedded complements}
In a different direction, applying Theorem \ref{thm:SMT for pseudo-metrics with negative holomorphic sectional curvature}  to the Kobayashi metric yields a \textit{Second Main Theorem} for pairs \((X,D)\) where \(X\setminus D\) is hyperbolically embedded in \(X\).  While this result is independent of  the other main applications of the present paper, it also goes in the direction of the Griffiths conjecture and we therefore felt that it was noteworthy to mention it here.  The algebraic counterpart of this statement was obtained by Pacienza and Rousseau \cite{PR07}. 

\begin{theorem}[{= Theorem \ref{thm:SMTHyp}}]
Let \(X\) be a smooth projective variety and let \(H\) be a reduced divisor on \(X\) such that \(X\setminus H\) is hyperbolically embedded in \(X\). Let \(A\) be  an ample line bundle on \(X\). There exists a constant \(\alpha >0\) such that for any non-compact parabolic Riemann surface $B$ equipped with a parabolic exhaustion function $\sigma$ and a non-constant holomorphic map \(f:B\to X\) such that \(f(B)\not\subset H\), there exists a positive real number $C$ such that the following inequality
\[T_{f,A}(r)\leqslant \alpha\big(N^{(1)}_{f,H}(r)-\mathfrak{X}_\sigma(r)\big)+C \cdot \log r\]
holds for all $r \in \R_{\geq 1}$ outside a subset of finite Lebesgue measure.
\end{theorem}

\subsection{Organization of the paper}
In Section \ref{sec:Nevanlinna} we give a proof of Theorem \ref{thm:SMT for pseudo-metrics with negative holomorphic sectional curvature} that does not assume any prior knowledge in Nevanlinna Theory. In Section \ref{sec:Hodge} we first recall the basic definitions from variational Hodge Theory and then prove Theorem \ref{thm:holomorphic sectional curvature for VHS}.
The proof of Theorem \ref{thm:SMT for VHS} is given in Section \ref{sec:proof SMT for VHS} and the applications (Theorem \ref{thm:criterion for Borel hyperbolicity} and Theorem \ref{thm:generalization of Nadel}) are given in Section \ref{sec:proof applications}. Finally, we prove Theorem \ref{thm:SMTHyp} in Section \ref{sec:SMTHyp}.

\section{Second Main Theorem for pseudo-metrics with negative holomorphic sectional curvature}\label{sec:Nevanlinna}
\subsection{Metric and curvature} Let us start by recalling some standard notations and definitions. Let \(B\) be a Riemann surface. A pseudo-metric on \(B\) is a map \(h:T_B\times_B T_B\to \C\) of class  \(\mathscr{C}^{\infty}\), such that for any \(x\in B\) the restricted map \(h_x:T_{B,x}\times T_{B ,x}\to \C\) is a symmetric sesquilinear form. To such a pseudo-metric, we associate the pseudo-norm \(\|\cdot\|_h:T_B\to \R^+\) defined by \(\|\xi\|^2_h=h(\xi,\xi)\) for any \(\xi \in T_B\). For any \(x\in B\), the restriction of \(\|\cdot\|_h\) to \(T_{B,x}\) will be denoted by \(\|\cdot\|_{h,x}:T_{B,x}\to \R^+\). The degeneracy set of the pseudo-metric \(h\) will be denoted by
\[\Sigma_h:=\left\{x\in B\ ; \ \|\cdot\|_{h,x}\equiv 0\right\}.\]
We will only consider pseudo-metrics  whose degeneracy set is a discrete subset of $B$.  
Locally, given a holomorphic coordinate \(z\) on some open subset  \(U\subset B\), letting \(\lambda : U\to \R^+\), the function 
\[\lambda(z)=\left\|\frac{\partial}{\partial z}\right\|^2_{h,z},\]
 one defines the associated \((1,1)\)-form to the pseudo-metric \(h\) to be locally defined by
\[\omega=\frac{i}{2}  \lambda  dz\wedge d\bar{z}.\]
This defines a global \((1,1)\)-form on \(B\).
The Ricci curvature of \(h\) or equivalently, the Ricci curvature of \(\omega\) is defined to be
 \[\Ric \omega=-2\pi dd^c\log \lambda=-i\partial\bar{\partial}\log\lambda = -i\frac{\partial^2\log\lambda}{\partial z\partial\bar{z}}dz\wedge d\bar{z}.\]
 This is a \((1,1)\)-form defined on \(B\setminus \Sigma_h\). The Gaussian curvature of \(h\) is the function \(K:B\setminus \Sigma_h\to \R\) such that\footnote{In this article, we use the \(d^c\) operator defined by \(d^c=\frac{i}{4\pi}(\bar{\partial}-\partial)\), so that \(dd^c=\frac{i}{2\pi}\partial\bar{\partial}\).}
\[\Ric \omega=K  \omega.\]

 Let  \(\Sigma\subset B\) be a reduced subset (i.e. \(\Sigma\) is a discrete subset of \(B\)). Given a point \(p\in \Sigma\) and a local coordinate \(z\) centered at \(p\), writing \(\omega=\frac{i}{2}  \lambda dz\wedge d\bar{z}\), we say that:
\begin{itemize}
\item \(\omega\) has a \emph{logarithmic singularity} at \(p\) if there exists \(C\in \R^+\) such that 
\[\lambda\leqslant \frac{C}{|z|^2}.\]
\item \(\omega\) has a \emph{Poincaré singularity} at \(p\) if there exists \(C\in \R^+\) such that 
\[\lambda\leqslant \frac{C}{|z|^2(\log|z|^2)^2}.\]
\end{itemize}
If \(\log \lambda\) is locally integrable, then one can define the current \([\log \lambda]\). Therefore, if \(\omega\) is such that in any local chart \(\log \lambda\) is locally integrable (in which case, we say abusively that \(\log \omega\) is locally integrable), one can define on \(B\) the current \[\Ric[\omega]=-2\pi dd^c[\log\lambda].\] As usual, this means that for any \(\mathscr{C}^{\infty}\) function with compact support \(\varphi\), one has 
\(\Ric[\omega](\varphi):=-2\pi\int_B (\log \lambda) dd^c\varphi.\) One easily checks that this is a well-defined current on \(B\).

\subsection{A lemma for Ricci currents}
Let \(B\) be a Riemann surface and \(\Sigma\) be a reduced divisor on \(B\). Let \(h\) be a pseudo-metric on \(B\setminus \Sigma\) and let \(\omega\) be the associated \((1,1)\)-form. If \(\log \omega\) is locally integrable, one can define as above \(\Ric[\omega]\). On the other hand  if the Ricci form \(\Ric \omega\) (defined on \(B\setminus (\Sigma\cup \Sigma_h)\)) is locally integrable (on \(B\)), one can define the current \([\Ric \omega]\) on \(B\). The first step is to compare \([\Ric \omega]\) and \(\Ric[\omega]\). This is based on the following elementary lemma concerning subharmonic functions whose proof is left to the reader. 

\begin{lemma}\label{lem:SH} Let \(\psi\) be a  subharmonic and \(\mathscr{C}^\infty\) function on \(\Delta^*\). Suppose that   \(\psi\) is bounded above.  Then \(\psi\) extends as a subharmonic function on \(\Delta\). Moreover, the \((1,1)\)-form \(dd^c\psi\) is locally integrable on \(\Delta\) and the following inequality holds in the sense of currents:
\[[dd^c\psi]\leqslant dd^c[\psi].\]
\end{lemma}

Applying this lemma on the pair \((B,\Sigma)\), one obtains the following.

\begin{lemma}\label{lem:RicciCurrents}
Let \(B\) be a Riemann surface with a reduced divisor \(\Sigma\). Let \(h\) be a pseudo-metric on \(B\setminus \Sigma\) and let \(\omega\) be the associated \((1,1)\)-form. Suppose that \(\omega\) has at most logarithmic singularities around every point of \(\Sigma\) and that \(\log \omega\) is locally integrable over \(B\). Suppose moreover that there exists a smooth \((1,1)\)-form \(\tilde\omega\) on \(B\) such that over \(B\setminus (\Sigma\cup \Sigma_h)\) one has \[\tilde{\omega}-\Ric \omega\geqslant 0.\]
Then:
\begin{enumerate}
\item Both currents \([\Ric \omega]\) and \(\Ric [\omega]\) are  well-defined,
\item And one has
\[[-\Ric \omega]\leqslant 2\pi [\Sigma]+(-\Ric[\omega]).\]
\end{enumerate}
\end{lemma}

\begin{proof} The statement is local, so we restrict ourselves to a contractible neighborhood \(V\) of a point \(p\in B\). In such a neighborhood, let $z$ be a local coordinate that vanishes at $p$ and write \(\omega=\frac{i}{2}\lambda dz\wedge d\bar{z}\). The function \(\log\lambda\) is locally integrable by hypothesis. Moreover, the function \(\varphi:z\mapsto \log(|z|^2 \lambda (z))\) is bounded above, since by hypothesis on the singularities of \(\omega\) the function \(z\mapsto |z|^2\lambda(z)\) is bounded above. 

Applying the \(dd^c\) lemma, we can suppose that \(\tilde{\omega}=dd^c\alpha\) for some function \(\alpha\) of class \(\mathscr{C}^\infty\) . 
Applying  Lemma \ref{lem:SH} with $\psi=\varphi+\alpha$, we obtain that \(dd^c\psi\) is locally integrable, therefore so is \(dd^c\varphi=dd^c\log\lambda\) outside \(p\). This shows that \([\Ric\omega]\) is well-defined. Moreover, we obtain
\[[dd^c\alpha]+[dd^c\varphi]=[dd^c\psi]\leqslant dd^c[\psi] = dd^c[\alpha]+dd^c[\varphi].\]
Since \(\alpha\) is of class \(\mathscr{C}^\infty\), one has \([dd^c\alpha]=dd^c[\alpha]\), hence
\[[-\Ric \omega]=2\pi [dd^c\varphi]\leqslant 2\pi dd^c[\varphi].\] Moreover the Lelong-Poincaré lemma implies that on \(V\), one has
\[dd^c[\varphi]=dd^c[\log(|z|^2\lambda)]=dd^c[\log|z|^2]+dd^c[\log\lambda]=[\Sigma]-\frac{1}{2\pi}\Ric[\omega].\]
Altogether, we obtain the expected result.
\end{proof}

\subsection{Preliminaries on parabolic Riemann surfaces} Let us first introduce some terminology concerning parabolic Riemann surfaces. A detailed presentation of parabolic Riemann surfaces can for instance be found in \cite{N-S70}. General references for Nevanlinna theory for parabolic manifolds are for instance the work of Stoll \cite{Sto77,Sto83}, and more recently the article \cite{P-S14} by P\u{a}un and Sibony. We have followed closely this last article in our proof of the \textit{Second Main Theorem}, with some changes in the notations to make the analogy with the algebraic situation more transparent.

\begin{definition} 
A non-compact Riemann surface $B$ is called parabolic if it admits a parabolic exhaustion function, i.e. a continuous proper function $\sigma : B \rightarrow [0, + \infty)$ such that $ \log \sigma$ is harmonic outside a compact subset of $B$.
\end{definition}

If \((B,\sigma)\) is a non-compact Riemann surface equipped with a parabolic exhaustion function, then, for any \(r > 0\),  the open pseudo-ball of radius \(r\) and the pseudo-sphere of radius \(r\) will be denoted by 
\[B(r)  :=\{p\in B\ ;\ \sigma(p)<r\},\ \text{and} \ S(r) :=\{p\in B \ ; \ \sigma(p)=r\}\]
respectively. When \(r\) is a regular value of \(\sigma\) (which is the case for almost all \(r\)), the pseudo-sphere \(S(r)\) is smooth, and one considers on it the measure 
\[d\mu_r:=d^c\log\sigma|_{S(r)}.\] 
\begin{example}
\begin{enumerate}
\item The most standard example of parabolic Riemann surface is \(\C\), endowed with the parabolic exhaustion function \(\sigma(z)=|z|\). In this case \(B(t)\) is the disc of radius \(t\) centered at zero, and \(d\mu_r=\frac{d\theta}{4\pi}\), where \(\theta\) denotes the argument of \(z\).
\item If there exists a proper finite morphism \(\pi:B\to \C\), then \(B\) is a parabolic Riemann surface and a parabolic exhaustion function is given by \(\sigma(p)=|\pi(p)|\).
\item The previous example implies in particular that any affine Riemann surface is parabolic.
\end{enumerate}
\end{example}

One has the following version of the Jensen's formula:
\begin{proposition}[Jensen's formula, c.f. {\cite[Proposition 3.1]{P-S14}}]\label{Jensen formula}
Let $B$ be a non-compact parabolic Riemann surface equipped with a parabolic exhaustion function $\sigma$. Let \(\varphi:B\to [-\infty,+\infty)\) be a function which locally near every point of \(B\) can be written as the difference of two subharmonic functions. Then, for any \(r>1\) large enough, one has
\[\int_{1}^r\frac{dt}{t}\int_{B(t)}dd^c[\varphi]=\int_{S(r)}\varphi d\mu_r+O(1).\] 
\end{proposition}
One defines the \emph{weighted Euler characteristic} of \((B,\sigma)\) to be the function 
\[\mathfrak{X}_\sigma(r):=\int_1^r\chi(B(t))\frac{dt}{t},\]
where \(\chi(B(t))\) is the Euler characteristic of \(B(t)\). We will  only need the following result.
\begin{proposition}[Compare with {\cite[ Proposition 3.3]{P-S14}}]\label{weighted Euler characteristic}
Same notations as above. If \(\xi\in \Gamma(B,T_B)\) is any never vanishing holomorphic vector field on \(B\) (such a vector field exists since \(B\) is non-compact),   then one has
\[\mathfrak{X}_\sigma(r)=- \int_{S(r)}\log|d\sigma(\xi)|^2d\mu_r+O(\log r).\]
\end{proposition}

\begin{remark}
The statement we give here slightly differs from the formula given in \cite{P-S14} because the formula given in \textit{loc.cit.} contains some typo. For completeness and for the reader's convenience, we provide here a more detailed version of the proof of this formula given in \cite{P-S14}. We emphasize that we claim no originality here, and that we simply fill in the technical details of the proof outlined in  \cite{P-S14}. 
\end{remark}

\begin{proof}
Since \(\xi\) is of type \((1,0)\), one has \(d_{\xi}\sigma=\partial_{\xi}\sigma\). Since \(\xi\) is holomorphic and \(\log \sigma\) is harmonic outside the compact subset \(\overline{B}(r_0)\) for a well-choosen $r_0$, it follows that \(\partial_\xi\log  \sigma\) is holomorphic outside  \(\overline{B}(r_0)\). A direct computation shows that \(\sigma \cdot \partial_{\xi}\log\sigma=\partial_{\xi}\sigma\). Therefore, \(\partial_{\xi}\sigma\) vanishes only when \(\partial_{\xi}\log\sigma\) does, and outside \(\overline{B}(r_0)\) this happens only on a discrete subset (hence of measure zero). From now on, we take \(r > r_0\) outside this subset. 

Consider the vector field \(v=\overline{\partial_{\xi}\sigma}\cdot \xi\), or more precisely the real vector field \(v_{\mathbb{R}}\)  associated to it via the isomorphisms \(T^{1,0}\equiv T_{\R}\). This vector field has isolated singularities outside \(\overline{B}(r_0)\), and, along the boundary \(S(r)\), is pointing in the outward normal direction. To see this, we can work locally, with a holomorphic coordinate \(z\) such that \(\xi=\frac{\partial }{\partial z}\). In this coordinate, one has \(\overline{\partial_{\xi}\sigma}\cdot \xi=\overline{\frac{\partial \sigma}{\partial z}}\frac{\partial }{\partial z}=\frac{1}{2}\left(\frac{\partial \sigma}{\partial x}+i\frac{\partial \sigma}{\partial y}\right)\frac{\partial }{\partial z}\), therefore under the isomorphism \(T^{1,0}\equiv T_{\R}\), the corresponding real vector field \(v\) is given by \(\frac{\partial \sigma}{\partial x}\frac{\partial }{\partial x}+\frac{\partial \sigma}{\partial y}\frac{\partial }{\partial y}\).  In particular \(d_v\sigma=\left(\frac{\partial \sigma}{\partial x}\right)^2+\left(\frac{\partial \sigma}{\partial y}\right)^2\geqslant 0\), and is strictly positive on the boundary \(S(r)\) with our choice of \(r\). This shows that \(v\) is pointing outwards \(B(r)\), and moreover it is immediate that it is tangent to the normal direction. 

Let \(\tilde{v}\) be a \(\mathscr{C}^\infty\) vector field on \(B\) with isolated zeros  which coincides with \(v_{\mathbb{R}}\) outside \(\overline{B}(r_0)\). Applying the Poincaré-Hopf index theorem to the surface with boundary \(\overline{B}(r) \) and the vector field \(\tilde{v}\), we get that 
\[\chi(B(r))=\sum_{p\in B(r)}{\rm index}_p(\tilde{v})=\sum_{p\in \overline{B}(r_0)}{\rm index}_p(\tilde{v})+\sum_{p\in B(r)\setminus  \overline{B}(r_0)}{\rm index}_p(v_{\mathbb{R}})\]
We claim that for every \(p\in B(r)\setminus\overline{B}(r_0)\), \({\rm index}_pv_{\mathbb{R}}=-{\rm ord}_p\partial_{\xi}\log\sigma\). To see this, recall that if \(f\) is a holomorphic function, then \({\rm index}_0\left(f\frac{\partial }{\partial z}\right)_{\mathbb{R}}={\rm ord}_0f\), where \(\left(f\frac{\partial }{\partial z}\right)_{\mathbb{R}}={\rm Re}(f)\frac{\partial}{\partial x}+{\rm Im}(f)\frac{\partial}{\partial y}\) is the real vector field associated to \(f\frac{\partial}{\partial z}\). Therefore, in our situation, \({\rm index}_p(\partial_{\xi}\log\sigma\cdot \xi)_{\mathbb{R}}={\rm ord}_p\partial_{\xi}\log\sigma \). On the other hand, 
 \[{\rm index}_p(\partial_{\xi}\log\sigma\cdot \xi)_{\mathbb{R}}= {\rm index}_p\left(\frac{1}{\sigma}\partial_{\xi}\sigma\cdot \xi\right)_{\mathbb{R}}={\rm index}_p\left(\partial_{\xi}\sigma\cdot \xi\right)_{\mathbb{R}}=-{\rm index}_p\left(\overline{\partial_{\xi}\sigma}\cdot \xi\right)_{\mathbb{R}}=-{\rm index}_pv_{\mathbb{R}}.\]
Applying now the Lelong-Poincaré formula, we obtain 
 \[-\sum_{p\in B(r)\setminus \overline{B}(r_0)}{\rm index}_p(v_{\mathbb{R}})=\sum_{p\in B(r)\setminus \overline{B}(r_0)}{\rm ord}_{p}\partial_{\xi}\log\sigma=\int_{B(r)\setminus \overline{B}(r_0)}dd^c[\log|\partial_{\xi}\log\sigma|^2],\]
 and therefore 
 \begin{eqnarray*}
 -\chi(B(r))&=&\int_{B(r)\setminus \overline{B}(r_0)}dd^c[\log|\partial_{\xi}\log\sigma|^2]+ C_1\label{eq:EulerNegative}\end{eqnarray*}
 where 
\[ C_1=-\sum_{p\in \overline{B}(r_0)}{\rm index}_p(\tilde{v}).\]
Let \(\varphi\) be a \(\mathscr{C}^\infty\) function on \(B\) that coincides with \(\log|\partial_{\xi}\log \sigma|^2\) over \(B\setminus \overline{B}(r_0)\), so that 
\[\int_{B(r)\setminus \overline{B}(r_0)}dd^c[\log|\partial_{\xi}\log\sigma|^2]=\int_{B(r)}dd^c[\varphi]-\int_{\overline{B}(r_0)}dd^c[\varphi].\]
Therefore, setting \(C_2=-\int_{\overline{B}(r_0)}dd^c[\varphi]\) and \(C=C_1+C_2\), Jensen's formula implies
\begin{eqnarray*}-\int_1^r\chi(B(t))\frac{dt}{t}=\int_1^r\left(\int_{B(t)}dd^c[\varphi]\right)\frac{dt}{t}+C\log r= \int_{S(r)}\varphi d\mu_r +C\log r+O(1). \end{eqnarray*}
On the other hand, one has 
\begin{eqnarray*}
\int_{S(r)}\varphi d\mu_r&=& \int_{S(r)}\log|\partial_{\xi}\log\sigma|^2d\mu_r=\int_{S(r)}\log\left|\frac{\partial_{\xi}\sigma}{\sigma}\right|^2d\mu_r\\
&=&\int_{S(r)}\log\left|\partial_{\xi}\sigma\right|^2d\mu_r-\int_{S(r)}\log\left|\sigma\right|^2d\mu_r\\
 &=&\int_{S(r)}\log\left|\partial_{\xi}\sigma\right|^2d\mu_r+2\log r\int_{S(r)}d\mu_r =\int_{S(r)}\log\left|d_{\xi}\sigma\right|^2d\mu_r+C_3\log r.\end{eqnarray*}
In the last line we used the fact \(\int_{S(r)}d\mu_r \) is independant of \(r\) (as an application of Stokes formula).
It remains to put all this together to obtain the announced result.
\end{proof}

\subsection{Notations from Nevanlinna theory}\label{ssec:Nevanlinna}
 Parabolic Riemann surfaces provide a suitable framework for value distribution theory. In this section we provide a brief account on this theory and refer the reader to the works \cite{SNM66, Wu70, Sto77, Sto83} or \cite{P-S14} for a more detailed presentation. \\
 
Let $B$ be a non-compact parabolic Riemann surface equipped with a parabolic exhaustion function $\sigma$.  For any current \(\alpha\) of type \((1,1)\) on \(B\), we define the \emph{characteristic function of \(\alpha\)} to be the function  \(T_{\alpha}:[1,+\infty)\to \R\) defined by
\[T_\alpha(r)=\int_1^r\left(\int_{B(t)}\alpha\right)\frac{dt}{t}.\]
If \(\Sigma\) is a discrete set of points in \(B\), we denote by \([\Sigma]\) the associated integration current on \(B\), and set 
 \[N_{\Sigma}:=T_{[\Sigma]}.\] 
 Suppose we are given a smooth projective variety \(X\) and a non-constant holomorphic map \(f:B\to X\).
 Let \(D\) be a reduced effective divisor on \(X\) such that \(f(B)\not\subset D\). The \emph{truncated counting function of \(f\) with respect to \(D\)} is the function 
 \[N^{[1]}_{f,D}(r)=N_{f^{-1}(D)}(r).\]
 We emphasize that \(f^{-1}(D)\) is the set theoretic inverse image, and therefore, this function doesn't take into account the intersection multiplicities. 
  Let \(L\) be a line bundle on \(X\), endow \(L\) with a \(\mathscr{C}^{\infty}\) hermitian metric \(h\) and denote by \(C_1(L,h)\) the first Chern form of the hermitian bundle \((L,h)\). Recall that \(C_1(L,h)\) is the \((1,1)\)-form locally defined by \(-dd^c\log\|s\|_h^2\), where \(s\) is any nowhere-vanishing local holomorphic section of \(L\). The \emph{characteristic function of \(f\) with respect to \(L\)} is the function 
 \[T_{f,L}(r):=T_{f^*C_1(L,h)}(r).\]
 This function depends on the choice of the metric \(h\) only up to a bounded function. Indeed, if \(h_1, h_2\) are two \(\mathscr{C}^{\infty}\) metrics on \(L\), then there exists a \(\mathscr{C}^{\infty}\) function \(\varphi:X\to \R_+^*\) such that 
 \[C_1(L,h_1)-C_1(L,h_2)=dd^c\varphi. \]
 By compactness of \(X\), the function \(\varphi\) is bounded and Jensen's formula implies that the function \(T_{f^*dd^c\varphi}\) is also bounded. Similar results can be obtained if one allows metrics with mild singularities. If \((L,h_1)\) is a hermitian line bundle, a \emph{singular metric \(h_2\)} is of the form \(h_2=e^{-\varphi} h_1\) where \(\varphi:X\to \R\cup\{\pm\infty\}\) is a locally integrable function. In this situation, we define the  \emph{curvature current of the line bundle \(L\) endowed with the singular metric \(h_2\)} to be \[C_1(L,h_2):=C_1(L,h_1)+dd^c[\varphi].\]
   \begin{lemma}
 Let \(X\) be a smooth projective variety. Let \(L\) be a line bundle on \(X\) and let \(h_1\) be a smooth hermitian metric on \(L\). Let  \(h_2\) be a singular metric on \(L\).   If the curvature current \(C_1(L,h_2)\) is positive in the sense of current, then  for any non-compact parabolic Riemann surface equipped with a parabolic exhaustion function and any non-constant holomorphic map \(f:B\to X\), one has
  \[T_{f^*C_1(L,h_2)}(r)\leqslant T_{f^*C_1(L,h_1)}(r)+ O(1).\]
  \end{lemma}

 \begin{proof}
Since the curvature current \(C_1(L,h_2)\) is positive in the sense of current, it follows that the function \(\varphi\) is quasi-plurisubharmonic. Recall that this means that locally, \(\varphi\) is the sum of a plurisubharmonic function and a smooth function. From this and the compactness of \(X\), it follows that \(\varphi\) is bounded above. One can now apply Jensen's formula to \(\varphi\) and we obtain
 \[T_{f^*C_1(L,h_2)}(r)-T_{f^*C_1(L,h_1)}(r)=\int_1^r\left(\int_{B(t)}dd^c\varphi\right)\frac{dt}{t}=\int_{S(r)}\varphi  d\mu_r+O(1)\leqslant O(1).\] 
 \end{proof}
 
 \begin{lemma}\label{lem:singular metric with log growth}  Let \(X\) be a smooth projective variety. Let \(L\) be a line bundle on \(X\) and let \(h_1\) be a smooth hermitian metric on \(L\). Let  \(h_2=e^{-\varphi}h_1\) be a singular metric on \(L\). Let \(D\) be an effective divisor on \(X\) and let \(s_D\in H^0(X,\mathscr{O}_X(D))\) be a global section such that \(D = \{s_D=0 \}\). Let \(\|\cdot\|_D\) be the norm associated to a hermitian metric on the line bundle \(\mathscr{O}_X(D)\). If locally on $X$ there exists \(\beta,C\in \mathbb{\R_+}\) such that 
 \[e^{-\varphi}\leqslant C \log(\|s_D\|_D^{-2\beta})\]
then for any ample line bundle  \(A\) on \(X\), any non-compact parabolic Riemann surface $B$ equipped with a parabolic exhaustion function and any non-constant holomorphic map \(f:B\to X\) such that \(f(B)\not\subset D\), one has
  \[T_{f^*C_1(L,h_1)}(r)\leqslant T_{f^*C_1(L,h_2)}(r)+O(\log( T_{f,A}(r))).\]
 \end{lemma}
 \begin{proof} Let \((U_i)_i\) be a finite set of open subsets of \(X\) on which the inequality of the hypothesis holds, let \(C_i,\beta_i\) be the associated  constants, and let \(C=\max\{C_i\}\) and \(\beta=\max\{\beta_i\}\). Let $B$ be a non-compact parabolic Riemann surface equipped with a parabolic exhaustion function and let \(f:B\to X\) be a non-constant holomorphic map such that \(f(B)\not\subset D\).
 The hypothesis implies that \(-\varphi\circ f\leqslant \log C+ \log(\beta)+\log (\log(\|s_D\circ f\|^{-2}_D))\). Applying Jensen's formula and the hypothesis one obtains
\begin{eqnarray*}
\int_1^r\left(\int_{B(t)}-f^*dd^c[\varphi]\right)\frac{dt}{t}&=&\int_{S(r)}\varphi\circ f d \mu_r+O(1)\leqslant \int_{S(r)}\log (\log(\|s_D\circ f\|^{-2}_D)) d\mu_r+O(1)\\
&\leqslant& \log\int_{S(r)}\log(\|s_D\circ f\|^{-2}_D)d\mu_r+O(1)=\log \left( T_{f^*dd^c[\log(\|s_D\|^{-2}_D)]}(r) \right)\\
&=&\log \big(-T_{[f^*D]}(r)+T_{f^*C_1(\OO_X(D),\|\cdot \|_D)}\big)\leqslant \log (T_{f,\OO_X(D)}(r))=O(\log T_{f,A}(r)).
\end{eqnarray*}
From this relation, the announced inequality follows immediately from the definition of the curvature of the singular metric \(h_2\).
 \end{proof}

Let us conclude this section by recalling some classical notations from Nevanlinna theory that shall be needed to state our results. Given two functions \(g,h : [1,+\infty) \to \R\) we write
 \[g(r)\leqslant_{\rm exc} h(r)\]
 if there exists a Borel subset \(E\subset [1,+\infty)\) of finite Lebesgue measure such that \(g(r)\leqslant h(r)\) for all \(r\in [1,+\infty)\setminus E\).
 Moreover, given \(s:[1,+\infty)\to \R\), we shall write \(g(t)\leqslant_{\rm exc}h(t)+O(s(t))\) if there exists \(C\in \R_+\) such that \(g(t)\leqslant_{\rm exc} h(t)+Cs(t)\). 
 
 This notation is motivated from the use of the Borel lemma  (see Lemma 1.2.1 in \cite{N-W14}) :
 \begin{lemma}[Borel's lemma] Let \(\varphi:[1,+\infty)\to \R\) a monotone increasing function. Then, for any \(\ep>0\),
 \[\varphi'(r)\leqslant_{\rm exc} \varphi(r)^{1+\ep}. \]
 \end{lemma}
 One of the main goal of Nevanlinna theory is to compare the characteristic functions and the counting functions for different line bundles and divisors. For instance, as a consequence of the \emph{First Main Theorem} in Nevanlinna theory, one obtains \emph{Nevanlinna's inequality}, which guarantees that, if \(X\) is a smooth projective variety and \(D\) is an effective divisor on \(X\), then, for any non-compact parabolic Riemann surface $B$ equipped with a parabolic exhaustion function $\sigma$ and for any non-constant holomorphic map \(f:B\to X\) such that 
 \(f(B)\not\subset D\), one has
 \[N_{f,D}(r)\leqslant T_{f,\OO_X(D)}(r)+O(1).\]

In an opposite direction, we say that a \emph{Second Main Theorem} holds for a pair \((X,D)\) and a line bundle \(L\) on \(X\) if there exists a constant \(\alpha\) and an ample line bundle \(A\) on \(X\) such that for any non-compact parabolic Riemann surface $B$ equipped with a parabolic exhaustion function $\sigma$ and any non-constant holomorphic map \(f:B\to X\) such that \(f(B)\not\subset D\), one has
\begin{equation*}T_{f,L}(r)\leqslant_{\rm exc}\alpha (N^{[1]}_{f,D}(r) - \mathfrak{X}_{\sigma}(r))+O(\log r+\log T_{f,A}(r)).\label{eq:SMTGeneralVersion}\end{equation*}
One easily checks that if this inequality holds for some ample line bundle \(A\), then it holds for any ample line bundle on \(X\). The term 
\(O(\log r+\log T_{f,A}(r))\) should be thought of as an error term.

\subsection{A version of the logarithmic derivative lemma}

\begin{lemma}\label{lem:RicciNegligeable}
Let $B$ be a non-compact parabolic Riemann surface equipped with a parabolic exhaustion function $\sigma$, and let \(\xi\in \Gamma(B,T_B)\) be a global trivialization for \(T_B\). Let \(\Sigma\) be a reduced divisor on \(B\). Let \(h\) be a pseudo-metric on \(B\setminus \Sigma\) such that the function \(\log \|\xi\|_h\) can be locally written as the difference of two subharmonic functions. Let \(\omega\) denotes the associated \((1,1)\)-form and suppose that \(\omega\) is locally integrable. Then \(\Ric[\omega]\) is well-defined and one has 
\begin{eqnarray*}T_{-\Ric[\omega]}(r)
& \leqslant_{\rm exc}& -2\pi \mathfrak{X}_{\sigma}(r)+O\left(\log r+ \log T_{[\omega]}(r)\right).\end{eqnarray*}
\end{lemma}
\begin{proof} Consider the function  
\[\varphi=\log(\|\xi\|^2_h)\]
on \(B\setminus \Sigma\), which by hypothesis is locally the difference of two subharmonic functions on \(B\).

 By definition, one has
\[\Ric[\omega]=-2\pi dd^c[\varphi].\]
By Jensen's formula (Proposition \ref{Jensen formula}), one has, for any large enough regular value \(r\) of \(\sigma\),
\[\frac{1}{2\pi}T_{-\Ric[\omega]}(r)=\int_1^{r}\left(\int_{B(t)}dd^c[\varphi]\right)\frac{dt}{t}=\int_{S(r)}\varphi d\mu_r+O(1)\]
 Therefore we are reduced to prove that
\[\int_{S(r)}\varphi d\mu_r= - \mathfrak{X}_\sigma(r)+O\left(\log r+ \log T_{[\omega]}(r)\right).\]
One the one hand, one has 
\begin{eqnarray*}\int_{S(r)}\varphi d\mu_r&=&\int_{S(r)}\log(\|\xi\|^2_h) d\mu_r=\int_{S(r)}\log\left(\frac{\|\xi\|^2_h|d_\xi(\sigma)|^2}{|d_\xi(\sigma)|^2}\right) d\mu_r\\
&=& \int_{S(r)}\log\left(\frac{\|\xi\|^2_h}{|d_\xi(\sigma)|^2}\right) d\mu_r+ \int_{S(r)}\log\left(|d_\xi(\sigma)|^2\right) d\mu_r
\end{eqnarray*}
On the other hand, it follows from Proposition \ref{weighted Euler characteristic} that 
\[ \int_{S(r)}\log\left(|d_\xi(\sigma)|^2\right) d\mu_r= - \mathfrak{X}_\sigma(r)+ O(\log(r)).\] 

It therefore remains to bound the other term. First observe that it follows from Fubini's theorem that for any smooth \(1\)-form \(\varphi\) on $B$, one has 
\begin{equation*}
\int_{B(r)}d\sigma\wedge \varphi=\int_{1}^r\left(\int_{S(t)}\varphi\right)dt.
\end{equation*}
Letting $\varphi = \frac{\|\xi\|^2_h}{|d_\xi(\sigma)|^2} d^c\sigma$ and differentiating with respect to $r$, we obtain that 
\[ \int_{S(r)}\frac{\|\xi\|^2_h}{|d_\xi(\sigma)|^2} d^c\sigma  = \frac{d}{dr}\int_{ B(r)}\frac{\|\xi\|^2_h}{|d_\xi(\sigma)|^2} d\sigma\wedge d^c\sigma.\]

Therefore, by the concavity of the logarithm, 
\begin{eqnarray*}
 \int_{S(r)}\log\left(\frac{\|\xi\|^2_h}{|d_\xi(\sigma)|^2}\right) d\mu_r&\leqslant & \log  \int_{S(r)}\frac{\|\xi\|^2_h}{|d_\xi(\sigma)|^2} d\mu_r= \log  \frac{1}{r}\int_{S(r)}\frac{\|\xi\|^2_h}{|d_\xi(\sigma)|^2} d^c\sigma\\
& =& \log  \frac{1}{r}\frac{d}{dr}\int_{ B(r)}\frac{\|\xi\|^2_h}{|d_\xi(\sigma)|^2} d\sigma\wedge d^c\sigma.
\end{eqnarray*}

On the other hand, one has
\[\frac{\|\xi\|^2}{|d_\xi(\sigma)|^2}d\sigma\wedge d^c\sigma= \frac{1}{\pi}\omega.\]
Indeed, if one takes a local coordinate \(z\) such that \(\xi=\frac{\partial}{\partial z}\) and we write \(\omega=\frac{i}{2}\lambda dz\wedge d\bar{z}\). One has \(\lambda=\|\xi\|^2_{h}\) and 
\(d_{\xi}\sigma=\frac{\partial \sigma}{\partial z}\). Moreover, since \(\sigma\) is real,
\[d\sigma\wedge d^c\sigma=\frac{i}{4\pi}\left(\frac{\partial \sigma}{\partial z}dz+\frac{\partial \sigma}{\partial \bar{z}}d\bar{z}\right)\left(\frac{\partial \sigma}{\partial \bar{z}}d\bar{z}-\frac{\partial \sigma}{\partial z}dz\right)=\frac{i}{2\pi}\frac{\partial \sigma}{\partial z}\frac{\partial \sigma}{\partial \bar{z}}dz\wedge d\bar{z}=\frac{i}{2\pi}\left|\frac{\partial \sigma}{\partial z}\right|^2dz\wedge d\bar{z}.\]
Hence
\[\frac{\|\xi\|^2}{|d_\xi(\sigma)|^2}d\sigma\wedge d^c\sigma= \frac{\lambda}{\left|\frac{\partial \sigma}{\partial z}\right|^2}\frac{i}{2\pi}\left|\frac{\partial \sigma}{\partial z}\right|^2dz\wedge d\bar{z}=\frac{i}{2\pi}\lambda dz\wedge d\bar{z}=\frac{1}{\pi}\omega.\]
By applying Borel's lemma twice, we obtain that for any \(\ep>0\), one has
\begin{eqnarray*}
  \frac{1}{\pi r}\frac{d}{dr}\int_{ B(r)}\omega&\leqslant_{{\rm exc} }& \frac{1}{\pi r}\left(\int_{ B(r)}\omega\right)^{1+\ep}=  \frac{r^\ep}{\pi}\left(\frac{1}{r}\int_{ B(r)}\omega\right)^{1+\ep}=\frac{r^\ep}{\pi}\left(\frac{d}{dr}T_{\omega}(r)\right)^{(1+\ep)}\\
  &\leqslant_{\rm exc}& \frac{r^\ep}{\pi}T_{\omega}(r)^{(1+\ep)^2}
\end{eqnarray*}
Altogether this implies that 
\[ \int_{S(r)}\log\left(\frac{\|\xi\|^2_h}{|d_\xi(\sigma)|^2}\right) d\mu_r \leqslant_{\rm exc} \log(\frac{r^\ep}{\pi}T_{\omega}(r)^{(1+\ep)^2})=\ep \log(r)+(1+\ep)^2\log T_{\omega}(r)-\log\pi,\]
which implies the result.
\end{proof}
Combining Lemma \ref{lem:RicciCurrents} and Lemma \ref{lem:RicciNegligeable} one obtains the following 
\begin{corollary}\label{cor:Tautological}
Let $B$ be a non-compact parabolic Riemann surface equipped with a parabolic exhaustion function $\sigma$, and let \(\xi\in \Gamma(B,T_B)\) be a global trivialization for \(T_B\). Let \(\Sigma\) be a reduced divisor on \(B\). Let \(h\) be a pseudo-metric on \(B\setminus \Sigma\) such that the function \(\log \|\xi\|_h\) can be locally written as the difference of two subharmonic functions and let \(\omega\) denotes the associated \((1,1)\)-form.  Suppose moreover that \(\omega\) has at most logarithmic singularities around points of \(\Sigma\). Then,
\begin{align*}T_{[-\Ric\omega]}(r)&\leqslant 2\pi N_\Sigma(r)+T_{-\Ric[\omega]}(r)+O(1)\\
&\leqslant_{{\rm exc}} 2\pi \big(N_\Sigma(r)- \mathfrak{X}_{\sigma}(r)\big)+O( \log r+\log T_{[\omega]}(r))\end{align*}
\end{corollary}
\begin{remark}
This result can be thought of as an intrinsic version of  \emph{McQuillan's tautological inequality}. Although it is not completely immediate, it is possible, with a suitable choice of metric, to derive from the previous corollary the usual version of the tautological inequality. Since this approach still requires some technical estimates and will not be further used in the present article,  we don't give the details here. We refer to \cite{Gas09} for a detailed presentation of the tautological inequality.
\end{remark}
\subsection{Second Main Theorem for metrics with negative holomorphic sectional curvature}
In this section, we observe that the hypothesis of Corollary \ref{cor:Tautological} are verified under an assumption on the Gaussian curvature of the pseudo-metric \(h\). 

\begin{lemma}\label{lem:HSC}
Let \(B\) be a non-compact Riemann surface endowed with a reduced divisor \(\Sigma\). Let \(h\) be a pseudo-metric on \(B\setminus \Sigma\) such that the holomorphic sectional curvature of \(h\) is bounded above by a negative constant \(-\gamma<0\) over \(B\setminus(\Sigma\cup\Sigma_h)\). Let \(\omega\) denote the \((1,1)\)-form associated to \(h\). Then:
\begin{enumerate}
\item The form \(\omega\) has at most a Poincaré singularity at the points of \(\Sigma\).
\item If \(\xi\in \Gamma(B,T_B)\) is a never-vanishing holomorphic vector field,  then \(\log\|\xi\|_h\) can locally be written  as the difference of two subharmonic functions.
\end{enumerate}
\end{lemma}

\begin{proof}
The first claim is an immediate consequence of the Schwarz lemma in view of our hypothesis on the Gaussian curvature.
To prove the second assertion, one can restrict ourself to small enough neighborhoods of points of \(\Sigma\). So fix \(p\in \Sigma\) and take a small enough neighborhood \(W\) of \(p\) with a local holomorphic coordinate \(z\) centered at \(p\) such that \(\xi=\frac{\partial }{\partial z}\) and such that \(\Sigma\cap W=\{p\}\). Set \(\lambda(z)=\|\xi(z)\|^2\) for all \(z\in W\). Since $\omega$ has a Poincar\'e singularity at $p$, it follows a fortiori that \(z \mapsto  |z|^2\lambda(z)\) is bounded above around $p$. On the other hand, by our curvature assumption, one has \(-\Ric\omega\geqslant \gamma \omega\) on \(W\setminus\{p\}\). Therefore the function \(z \mapsto \log(|z|^2\lambda(z))\) is subharmonic on \(W\setminus \{p\}\), and thus extends as a subharmonic function \(\varphi\) on \(W\). It follows that \(z \mapsto \log\lambda(z)=\varphi(z) -\log|z|^2\) is indeed the difference of two subharmonic functions on \(W\).
\end{proof}

\begin{theorem}\label{thm:SMTlocal} Let $B$ be a non-compact parabolic Riemann surface equipped with a parabolic exhaustion function and a reduced divisor \(\Sigma \subset B\). Let \(h\) be a pseudo-metric on \(B\setminus \Sigma\) such that the Gaussian curvature of \(h\) is bounded above by a negative constant \(-\gamma<0\) over \(B\setminus(\Sigma\cup\Sigma_h)\). If \(\omega\) denotes the \((1,1)\)-form associated to \(h\), then \(\omega\) is locally integrable on \(B\) and one has
\[T_{[\omega]}(r)\leqslant_{{\rm exc}} \frac{2\pi}{\gamma}\big(N_\Sigma(r) -\mathfrak{X}_{\sigma}(r)\big)+O(\log r+\log T_{[\omega]}(r)).\]
\end{theorem}
\begin{proof} Observe first that both currents \([\omega]\) and \([\Ric \omega]\) are well-defined in view of Lemma  \ref{lem:RicciCurrents} and Lemma \ref{lem:HSC}. By hypothesis, the Gaussian curvature \(K_h\) of \(h\) satisfies \(K_h\leqslant -\gamma\). Therefore, outside \(\Sigma_h\), one has
\[\Ric \omega=K_h\omega\leqslant -\gamma \omega\]
and thus, \(\gamma \omega\leqslant -\Ric\omega\).
At the level of currents,  it follows that 
\[[\omega]\leqslant \frac{1}{\gamma}[-\Ric\omega].\]
From this, one obtains
\[ T_{[\omega]}(r)\leqslant \frac{1}{\gamma}T_{[-\Ric\omega]}(r).\]
It suffices then to apply Corollary  \ref{cor:Tautological}.
\end{proof}

From this result one immediately infers a \textit{Second Main Theorem} for varieties having a metric with negative holomorphic sectional curvature. Let us first recall a definition. Let \(U\) be a complex manifold equipped with a Finsler pseudo-metric \(h\) of class \(\mathscr{C}^{\infty}\). Let $x \in U$ in the complementary of the degeneracy set of $h$. The holomorphic sectional curvature of \(h\) at \(x\) in a direction \(\xi\in T_{U,x}\setminus\{0\}\) is defined to be 
\[{\rm HSC}_{x,\xi}(h):=\sup K_{f^*h}(0)\]
where the supremum is taken over all holomorphic maps \(f:\Delta\to U\) such that \(f(0)=x\) and such that \(\xi\in \C\cdot f'(0)\). If \({\rm HSC}(\omega)\leqslant -\gamma\) for some \(\gamma>0\), we say that the holomorphic sectional curvature is bounded above by \(-\gamma\). In that case, if \(B\) is a Riemann surface and \(f:B\to U\) is a holomorphic map such that \(f^*h\) is a pseudo-metric on \(B\), then its Gaussian curvature is bounded above by \(-\gamma\). Therefore, we obtain the following

\begin{theorem}\label{thm:SMTglobal}
Let \(X\) be a smooth projective variety and let \(D\) be a reduced  divisor on \(X\). Suppose that \(X \setminus D\) is endowed with a Finsler pseudo-metric $h$ of class \(\mathscr{C}^{\infty}\) with holomorphic sectional curvature bounded above by \(-\gamma<0\). Assume that the degeneracy set of $h$ is contained in a nowhere-dense closed analytic subset of $X \setminus D$. Let $B$ be a non-compact parabolic Riemann surface equipped with a parabolic exhaustion function $\sigma$ and a non-constant holomorphic map \(f:B\to X\) whose image $f(B)$ is not contained in $D$ nor in the degeneracy set of $h$.  If \(\omega\) denotes the \((1,1)\)-form associated to the induced pseudo-metric $f^\ast h$ on \(B\setminus (f^{-1} (D))_{\rm red}\), then \(\omega\) is locally integrable on \(B\) and one has
\begin{equation*}
T_{[\omega]}(r)\leqslant_{{\rm exc}} \frac{2\pi}{\gamma}\big(N^{[1]}_{f,D}(r) -\mathfrak{X}_{\sigma}(r)\big)+O(\log r+\log T_{[\omega]}(r)).\label{eq:SMTglobal}
\end{equation*}
\end{theorem}

In the following formulation, the error term does not depend on \(\omega\).
\begin{corollary}
Same setting as in Theorem \ref{thm:SMTglobal}. Then, for any \(\ep >0\), one has 
\[(1-\ep)T_{[\omega]}(r)\leqslant_{{\rm exc}} \frac{2\pi}{\gamma}\big(N_{f,D}^{[1]}(r)-\mathfrak{X}_{\sigma}(r)\big)+O(\log r).\]
\end{corollary}
\begin{proof}
Just observe that the function \(T_{[\omega]}\) is unbounded and increasing, hence for any \(C, \ep>0\) we have
\[C \log T_{[\omega]}(r)\leqslant_{\rm exc} \ep T_{[\omega]}(r).\]
\end{proof}

\section{Recollection in Hodge theory}\label{Recollection in Hodge theory}\label{sec:Hodge}

\subsection{Complex polarized Hodge structure}
A complex polarized Hodge structure (of weight zero) on a finite-dimensional complex vector space $V$ is the data of a non-degenerate hermitian form $h$ on $V$ (the polarization) and of a decomposition $V = \bigoplus_{p \in \bZ} {V}^p$ (the Hodge decomposition) which is orthogonal for $h$ and such that the restriction of $h$ to $ {V}^p$ is positive definite for $p$ even and negative definite for $p$ odd. The $r_p := \dim V^p$ are called the Hodge numbers. The Hodge metric on $V$ is the positive-definite hermitian metric $h_H$ obtained from $h$ by imposing that the Hodge decomposition $V = \bigoplus_{p \in \bZ} {V}^p$ is $h_H$-orthogonal and setting $h_H := (-1)^p \cdot h$ on $V^p$.\\

The Hodge filtration is the decreasing filtration $\{F^{\sbt}\}$ on $V$ defined by $F^p := \bigoplus_{q \geq p} V^q$. The Hodge decomposition is determined by the Hodge filtration thanks to the formula $V^p = F^p \cap (F^{p+1})^\perp$. Here $(F^{p+1})^\perp$ denotes the orthogonal with respect to the polarization $h$, which is clearly equal to the orthogonal of $F^{p+1}$ with respect to the Hodge metric $h_H$.\\

If $[a, b]$ is the smallest interval such that $V^p = 0$ for $p \notin [a, b]$, then the length of $V$ is by definition the integer $w= b - a$. Note that by shifting the numbering of the Hodge decomposition, one can always assume that $a = 0$ and $b = w$.\\

The category of complex polarized Hodge structures is Abelian, semisimple and admits tensor products and internal Hom. In particular, $\End(V)$ inherits a complex polarized Hodge structure with decomposition $\End(V) = \bigoplus_{p \in \bZ} {\End(V)}^p$ where $\End(V)^p = \bigoplus_{s - r = p } \Hom(V^r, V^s)$.\\

We often denote abusively by the same symbols $h$ or $h_H$ the metric that they induce on any object obtained from $V$ by a tensorial construction. 
\subsection{Variation of Hodge structures} 
A variation of complex polarized Hodge structures $\bV$ on a (reduced) complex analytic space $S$ consists in a complex local system $\cL$ on $S$ equipped with a non-degenerate hermitian form $h : \cL \otimes_{\bC} \bar \cL \rightarrow \underline{\bC}_S$ and a locally split finite filtration $\cF^{\sbt}$ of $\cV := \cL \otimes_\bC \cO_S$ by analytic coherent subsheaves such that 

\begin{itemize}
\item for every $s \in S$, the triple $(\cL_s, \cF^{\sbt}_s, h_s)$ defines a complex polarized Hodge structure;
\item (Griffiths' transversality) letting $\nabla := id \otimes d$, we have $\nabla( \cF^p) \subset \cF^{p-1} \otimes_{\cO_S} \Omega_S^1$ for all $p$.
\end{itemize}  

Note that when $S$ is smooth, the data of the complex local system $\cL$ is equivalent to the data of the holomorphic vector bundle $\cV := \cL \otimes_\bC \cO_S$ equipped with the integrable connection $\nabla := id \otimes d$.

The length of the complex polarized Hodge structure on the complex vector space $\mathcal{L}_s$ is independent of the point $s \in S$ and is called the length of $\bV$.

\subsection{Example} If $f : X \to S$ is a smooth projective holomorphic map, then the choice of a relatively ample line bundle on $X$ endows canonically the complex local system $R^k f_\ast \underline{\mathbb{C}}_X $ with a structure of variation of complex polarized Hodge structures for every $k \geq 0$.

\subsection{Associated system of Hodge bundles}
Starting with a variation of complex polarized Hodge structures $(\mathcal{V}, \nabla, \cF^{\sbt}, h)$ on a complex manifold $S$, we define $\cE^p := \cF^p / \cF^{p+1}$ and $\cE := \bigoplus_p \cE^p$. It follows from Griffiths transversality that $\nabla$ induces an $\cO_S$-linear morphism  $ \cE^p \rightarrow  \Omega^1_S \otimes_{\cO_S} \cE^{p-1}$ for every $p$. We denote by $\phi_p$ the corresponding element of $ \Omega^1_S \otimes_{\cO_S} \mathcal{H}om(\cE^p , \cE^{p-1})$ and set $\phi := \oplus_p \phi_p \in \Omega^1_S (\mathcal{E}nd(\cE))$. The pair $(\cE, \phi)$ together with the decompositions $\cE := \bigoplus_p \cE^p$ and $\phi := \oplus_p \phi_p $ is called the system of Hodge bundles associated to the variation of complex polarized Hodge structures $(\mathcal{V}, \nabla, \cF^{\sbt}, h)$. In what follows, we will often consider the so-called Higgs field $\phi \in  \Omega^1_S (\mathcal{E}nd(\cE))$ as an $\cO_S$-linear morphism $ \mathcal{T}_S \rightarrow \mathcal{E}nd(\cE)$ that we will abusively denote by the same symbol, and similarly for the $\phi_p$'s.\\

The holomorphic vector bundle $\cE$ comes equipped with the positive-definite hermitian metric $h_H$ that we call the Hodge metric. The decomposition $\cE := \bigoplus_p \cE^p$ is orthogonal for $h_H$.
The curvature of $(\cE, h_H)$ has been computed by Griffiths (the real structure plays no role in the computations):

\begin{proposition}[Griffiths, cf. {\cite[Theorem 6.2]{GriffithsIII}}. See also {\cite[Lemma 7.18]{Schmid}}]\label{curvature computation}
The curvature form $\Theta$ of $(\cE, h_H)$ satisfies
\[  h_{H}(\Theta_{(\cE, h_H)}(X, \bar Y) e, f) = h_{H}( \phi(X)e, \phi(Y) f)  - h_{H}(\phi^\ast(\bar Y) e, \phi^\ast (\bar X) f)  \]
for any vector fields $X$ and $Y$ of type $(1,0)$ and any sections $e$ and $f$ of $\cE$.
\end{proposition}

In the formula above, $\phi^\ast$ denotes the adjoint of $\phi$, meaning that $\phi^\ast(X)$ and $\phi(\bar X)$ are adjoint with respect to the induced Hodge metric on $\cE nd (\cE)$, for every tangent vector $X$ of type $(1,0)$. The decomposition $\phi = \oplus_p \phi_p $ corresponds to the decomposition $\phi^\ast = \oplus_p \phi^\ast_p $ where $\phi^\ast_p$ is a $(0,1)$-form with values in $\cH om(\cE^{p-1}, \cE^p)$.

\subsection{Induced system of Hodge bundles on the endomorphism bundle}

Let $\mathbb{V}$ be a variation of complex polarized Hodge structures on a complex manifold $S$, and let $(\cE = \bigoplus_p \cE^p, \phi := \oplus_p \phi_p )$ be its associated system of Hodge bundles. By functoriality, the holomorphic vector bundle $\cE nd(\cE)$ has a structure of system of Hodge bundles $\mathcal{E}nd(\cE) = \bigoplus_p \mathcal{E}nd(\cE)^p$ with Higgs field $\Phi$, which is associated to the variation of complex polarized Hodge structures $\End(\mathbb{V})$. One has $\mathcal{E}nd(\cE)^p = \bigoplus_k \mathcal{H}om(\cE^k , \cE^{k+p})$, and for any local holomorphic section $\Psi$ of $\mathcal{E}nd(\cE)$ and any local holomorphic section $X$ of $  \mathcal{T}_S$, we have 
 \[  \Phi(X)(\Psi)= \phi(X) \circ \Psi - \Psi \circ \phi(X)= [\phi(X) , \Psi] .\]
Similarly, the adjoint $\Phi^\ast$ of $\Phi$ is given by 
 \[  \Phi^\ast(X)(\Psi)=  \phi^\ast(X) \circ \Psi - \Psi \circ \phi^\ast(X)=  [\phi^\ast(X) , \Psi] .\]
In particular, we observe the following immediate application of Proposition \ref{curvature computation}.
\begin{proposition}\label{bisectional curvature}
For any any local holomorphic section $X$ of $  \mathcal{T}_S$ and any local holomorphic section $\Psi$ of $\mathcal{E}nd(\cE)$ such that $  \Phi(X)(\Psi)= 0$, we have 
\[  h_{H}(\Theta_{(\cE nd(\cE), h_H)}(X, \bar X) \Psi, \Psi) = - | [ \phi^\ast(\bar X), \Psi ] |_{h_H}^2  . \]
\end{proposition}
Note that this applies in particular to  $\Psi = \phi(Y)$ for any holomorphic vector field $Y$ since $[\phi(X) , \phi(Y)] = 0$.

\subsection{The Griffiths-Schmid pseudo-metric and its curvature}

So far the computations hold more generally for any Higgs bundle $(\cE, \phi)$ equipped with a harmonic metric $h_H$. We will now take advantage of the decomposition $\cE = \bigoplus_p \cE^p$.

\begin{definition}
For any integer $p$, we denote by $h_p$ the pseudo-metric on $  \mathcal{T}_S$ obtained as the pull-back by $\phi_p$ of the Hodge metric on $\mathcal{H}om(\cE^p , \cE^{p-1})$. The sum of these pseudo-metrics, which coincides with the pseudo-metric on $  \mathcal{T}_S$ obtained as the pull-back by $\phi$ of the Hodge metric on $\mathcal{E}nd(\cE)$, is called the Griffiths-Schmid pseudo-metric on $S$ induced by $\bV$.  
\end{definition}

Note that by its very definition, the locus where the Griffiths-Schmid pseudo-metric is non-degenerate coincide with the Zariski-open (but maybe empty) subset of $S$ where the $\cO_S$-linear morphism $\phi : \cT_S \rightarrow \cE nd(\cE)$ is injective, or in other words with the locus where the period map associated to $\bV$ is immersive.\\

Recall that one gets a correspondence by associating to any real $(1,1)$-form $\omega$ on $S$ the hermitian sesquilinear form $h$ on $\cT_S$ which satisfies $h(X,X) = - i \cdot \omega(X, \bar X)$ for any tangent vector of type $(1,0)$ (as usual we identify the holomorphic bundle with $T_{\bC}^{1,0}$). The following result is an easy consequence of Griffiths computation of the curvature of the Hodge metrics.
\begin{proposition}[Griffiths, cf. \cite{Schmid}]
Through the correspondence between real $(1,1)$-form and hermitian sesquilinear forms on $\cT_S$, the curvature form of the Chern connection of $(\det \cF^p , h_H)$ corresponds to $h_p$. A fortiori, the curvature form of the Chern connection of the Griffiths line bundle $L_{\mathbb{V}} :=  \otimes_p \det \cF^p$ equipped with the Hodge metric corresponds to the Griffiths-Schmid pseudo-metric.
\end{proposition}

In other words, the Griffiths-Schmid pseudo-metric on $S$ is K\"ahler and its K\"ahler form is equal to the Chern curvature form of $(L_{\mathbb{V}}  , h_H)$.

\begin{proposition}
Over the locus where it is non-degenerate (which is Zariski-open but possibly empty), the Griffiths-Schmid metric has non-positive holomorphic bisectional curvature and negative holomorphic sectional curvature $-  \gamma$ with 
\[ \frac{1}{\gamma}  \leqslant \frac{w^2}{4} \cdot \rank( \cV). \] 
\end{proposition}
\begin{proof}
We restrict ourselves to the locus in $S$ where the period map is immersive. It follows readily from Proposition \ref{bisectional curvature} that for any tangent vectors $X$ and $Y$ of type $(1,0)$, one has
\[  h_{H}(\Theta_{(\cE nd(\cE), h_H)}(X, \bar X) \phi(Y), \phi(Y)) = - | [ \phi^\ast(\bar X), \phi(Y) ] |_{h_H}^2   .\]
Since the Griffiths curvature decreases in subbundles, it follows that the Griffiths-Schmid metric has non-positive holomorphic bisectional curvature. Moreover, it follows that its holomorphic sectional curvature in the direction $X$ is less or equal to:
\[  \frac{1}{|X|_{h_H}^4} \cdot  h_{H}(\Theta_{(\cE nd(\cE), h_H)}(X, \bar X) \phi(X), \phi(X)) =- \frac{1}{|X|_{h_H}^4} \cdot | [ \phi^\ast(\bar X), \phi(X) ] |_{h_H}^2  . \]
Therefore, we can rely on the next result to conclude the proof.
\end{proof}

\begin{lemma}\label{computation sectional curvature}
If $(V, F^p ,h)$ is a complex polarized Hodge structure of length $w$, then for any $u \in  \End(V)^{-1} $:
\[  | [u, u^\ast] |^2_{h_H}    \geq   \frac{4}{w^2 \cdot \dim V} \cdot  | u|^4_{h_H} .\]
Here $u ^\ast$ denotes the adjoint of $u$ with respect to the Hodge metric $h_H$.
\end{lemma}

Note that when $(V, F^p ,h)$ is autodual (this holds for example if it is a real polarized Hodge structure), then $ \frac{4}{w^2 \cdot \dim V} =   \frac{2}{w \cdot \sum_{i= 1}^{w} \rank( F^p) }$. \\

Before giving the proof of the lemma, we recall two well-known elementary applications of the Cauchy-Schwarz lemma.

\begin{lemma} \label{Cauchy-Schwarz}
If $V$ is complex vector space equipped with a positive-definite hermitian sesquilinear form $h$ and $u = u^\ast \in \End(V)$, then the following inequality holds 
\[ \dim(V) \cdot \tr(u^2) \geq \left( \tr(u) \right)^2, \]
with equality exactly when $u$ is an homothety. 
\end{lemma}
\begin{proof}
This is a direct application of the Cauchy-Schwarz inequality 
\[ \left( \tr(u v^\ast) \right) ^2 \leq \tr(u u ^\ast) \cdot \tr(v v^\ast) \]
 with $v = id$. 
\end{proof}

\begin{lemma}\label{inequality}
For any real numbers $a_1, \ldots, a_n$ and any positive real numbers $\lambda_1, \ldots, \lambda_n$, the following inequality holds:
\[   \sum_i  \frac{a_i^2}{\lambda_i} \geq   \frac{1}{\sum_i \lambda_i} \cdot  \left( \sum_i a_i \right)^2, \]
with equality if and only if the $\frac{a_1}{\lambda_1} = \frac{a_2}{\lambda_2}= \cdots = \frac{a_n}{\lambda_n}$.
\end{lemma}

\begin{proof}[Proof of Lemma \ref{computation sectional curvature}]
In view of the expression of the Hodge metric on $\End(V)$, we need to prove the following inequality:

\[   \tr \left([u, u^\ast]^2 \right)   \geq   \frac{4}{w^2 \cdot \dim V} \cdot  \left( \tr \left(u \circ u^\ast \right) \right) ^2 .\]

Let $r_i := \dim(V^i)$ for every integer $i$. Up to renumbering, one can assume that $r_i = 0$ if $i < 0$ or $i > w$. If one writes $ u = \oplus_{i} u_i$ according to the decomposition 
\[ \End (V)^{-1} = \bigoplus_{i} \Hom \left(V^i, V^{i-1} \right), \]
so that $u_i = 0$ if $i >w$ or $i< 1$, we have that $ [u, u^\ast] = \oplus_{i}  \left(u_{i+1} \circ u_{i+1}^\ast - u_{i}^\ast \circ u_{i} \right)$. Since the decomposition above is orthogonal, it follows that
\[   \tr \left( [u, u^\ast]^2 \right) = \sum \limits_{i=1}^w \tr \left( \left(u_{i+1} \circ u_{i+1}^\ast - u_{i}^\ast \circ u_{i} \right)^2 \right) \geq  \sum_{i=1}^w \frac{1}{r_i} \cdot \left( \tr \left(u_{i+1} \circ u_{i+1}^\ast - u_i^\ast \circ u_i \right) \right)^2 ,\]
where the last inequality is obtained by applying Lemma \ref{Cauchy-Schwarz}. 
Letting $a_i := \tr \left( u_i \circ  u_i^\ast \right) = \tr \left( u_{i}^\ast \circ u_{i} \right)$, we get the inequality 
\[   \tr \left( [u, u^\ast]^2 \right)   \geq  \sum_{i=1}^w \frac{1}{r_i} \cdot \left( a_{i+1} - a_i \right)^2 . \]

On the other hand, with the help of Lemma \ref{inequality}, we have for any integer $p$:
\[   \sum_{i \geq p} \frac{1}{r_i} \cdot \left(a_{i+1} - a_i \right)^2 \geq  \frac{1}{ \sum_{i \geq p} r_i} \cdot \left( \sum_{i \geq p} (a_{i+1} - a_i) \right)^2 = \frac{a_p^2}{f_p} \]
and
\[   \sum_{i < p} \frac{1}{r_i} \cdot \left(a_{i+1} - a_i \right)^2 \geq  \frac{1}{ \sum_{i < p} r_i} \cdot \left( \sum_{i < p} \left( a_{i+1} - a_i \right) \right)^2 = \frac{a_p^2}{\dim V - f_p} ,\]
hence by summing the two we get:
\[ \sum_{i=1}^w \frac{1}{r_i} \cdot \left( a_{i+1} - a_i \right)^2  \geq  \left( \frac{1}{f_p}  + \frac{1}{\dim V - f_p} \right) \cdot a_p^2 . \]
Another application of Lemma \ref{inequality} shows that 
\[ \frac{1}{f_p}  + \frac{1}{\dim V - f_p} \geq  \frac{4}{\dim V} , \]
and by letting $p$ varying between $1$ and $w$, we get that
\[ w \cdot \left( \sum_{i=1}^w \frac{1}{r_i} \cdot (a_{i+1} - a_i)^2 \right)  \geq  \frac{4}{\dim V}  \cdot \sum_{p = 1}^w  a_p^2   . \]
A final application of Lemma \ref{inequality} gives that 
\[ \sum_{p = 1}^w  a_p^2 \geq \frac{1}{w} \cdot \left( \sum_{p = 1}^w  a_p\right) ^2 ,\]
so we finally get that
\[ \sum_{i=1}^w \frac{1}{r_i} \cdot \left( a_{i+1} - a_i \right)^2   \geq \frac{4}{w^2 \cdot \dim V} \cdot \left( \sum_{p = 1}^w  a_p\right) ^2 \]
and this finishes the proof.
\end{proof}

\section{Proof of the main Theorem}\label{sec:proof SMT for VHS}
In this section we prove Theorem \ref{thm:SMT for VHS}. As a first step, by putting together Theorem \ref{thm:SMT for pseudo-metrics with negative holomorphic sectional curvature} and Theorem \ref{thm:holomorphic sectional curvature for VHS}, we immediately deduce the following Arakelov-Nevanlinna inequality.
\begin{theorem}[An Arakelov-Nevanlinna inequality]\label{thm:Arakelov-Nevanlinna inequality} 
Let $B$ be a non-compact parabolic Riemann surface equipped with a parabolic exhaustion function $\sigma$, $\Sigma \subset B$ a discrete subset of points and $\mathbb{V} = (\cL, \cF^{\sbt}, h)$ be a variation of complex polarized Hodge structures of length $w$ on $ B - \Sigma$ with a non-constant period map. Then the first Chern form of the holomorphic line bundle $ L_{\mathbb{V}} = \otimes_p \det  \cF^p$ equipped with the hermitian metric induced by $h$ extends as a current \( [L_{\mathbb{V}}  ] \) on $B$, and there exists a positive real number $C$ such that the following inequality  
 \[ T_{[L_{\mathbb{V}} ]}(r) \leqslant \frac{w^2 \cdot \rk \cL}{4} \cdot ( - \mathfrak{X}_{\sigma}(r) + N_{\Sigma}(r))+  C \cdot (\log T_{[L_{\mathbb{V}}]}(r)+ \log r) \ \  \]
holds for all $r \in \bR_{\geq 1}$ outside a subset of finite Lebesgue measure.
\end{theorem}

Note that in this statement we make no assumption on the monodromy of $\cL$.\\

In view of Theorem \ref{thm:Arakelov-Nevanlinna inequality}, the proof of Theorem \ref{thm:SMT for VHS} is now a consequence of the following result.

\begin{proposition}
Let $X$ be a smooth projective complex algebraic variety and $\mathbb{V} = (\cL, \cF^{\sbt}, h)$ be a variation of complex polarized Hodge structures defined on the complementary of a normal crossing divisor $D \subset X$. Assume that $\cL$ has unipotent monodromies around the irreducible components of $D$. We denote by $\bar \cF^p$ the canonical Deligne-Schmid extension of $\cF^p$ to $X$ for any integer $p$ and by $\bar L_{\mathbb{V}} = \otimes_p \det \bar \cF^p$ the canonical extension of the Griffiths line bundle of $\mathbb{V}$.
Let $A$ be an ample line bundle on \( X\), $B$ be a non-compact parabolic Riemann surface equipped with a parabolic exhaustion function $\sigma$ and \(f:B\to X\) be a non-constant holomorphic map  such that \(f(B)\not\subset D\). Then the first Chern form of the line bundle $ L_{\mathbb{V}}$ equipped with the hermitian metric induced by $h$ extends as a current \( [L_{\mathbb{V}}  ] \) on $B$, and for any choice of a smooth hermitian metric on $\bar L_{\mathbb{V}} $ used to compute its Nevanlinna characteristic function, one has
  \[ T_{f, \bar L_{\mathbb{V}} }(r)  \leqslant T_{[L_{\mathbb{V}} ]}(r) +O(\log( T_{f,A}(r))).\]

\end{proposition}
\begin{proof}
This follows immediately from Lemma \ref{lem:singular metric with log growth}, since the assumption that $\cL$ has unipotent monodromies around the irreducible components of $D$ ensures that the (singular) Hodge metric on $\bar L_{\mathbb{V}} $ satisfies the growth assumption of the lemma. 
\end{proof}

\section{Borel hyperbolicity}\label{sec:proof applications}
In this section we prove Theorem \ref{thm:criterion for Borel hyperbolicity} and Theorem \ref{thm:generalization of Nadel}. We keep the notations of the statements.
\begin{proof}[Proof of Theorem \ref{thm:criterion for Borel hyperbolicity}]
Thanks to the main result of \cite{Ariyan-Robert}, it is sufficient to prove that for any smooth complex algebraic curve $C$, any non-constant holomorphic map $f : C^{an} \rightarrow { X}^{an}$ such that $f (C^{an}) \not \subset (D \cup \B_{+})^{an}$ is algebraic. 

Since $f (C^{an}) \not \subset (\B_{+})^{an}$, there exists an ample $\bQ$-divisor $A$ on $X$ and a section of some power of $ \bar L_{\mathbb{V}} (- (  \frac{w^2}{4} \cdot \rk \cL)  \cdot D -A)$ that does not vanish on $f (C^{an})$. With the help of the \textit{First Main Theorem}, it follows that for every $r \geq 1$:
 \[   T_{f, A}(r) \leqslant  T_{f,  \bar L_{\mathbb{V}} (- (  \frac{w^2}{4} \cdot \rk \cL)  \cdot D)}(r) \leqslant  T_{f,  \bar L_{\mathbb{V}}}(r) -  \frac{w^2 \cdot \rk \cL}{4} \cdot N^{[1]}_{f, D}(r) + O(1) .    \]
On the other hand, by applying Theorem \ref{thm:SMT for VHS} to $C^{an}$ equipped with the parabolic exhaustion function $\sigma$ constructed from a proper algebraic map $C \rightarrow \mathbb{A}^1$, we obtain that 
 \[  T_{f,  \bar L_{\mathbb{V}}}(r) -  \frac{w^2 \cdot \rk \cL}{4} \cdot N^{[1]}_{f, D}(r) \leqslant_{{\rm exc}} \frac{w^2 \cdot \rk \cL}{4} \cdot \left(-  \mathfrak{X}_{B,\sigma}(r) \right) + O \left( \log r + \log T_{f, A}(r) \right) . \ \  \]
 Therefore, we get from the two preceding inequalities that
 \[  T_{f, A}(r) \leqslant_{{\rm exc}} \frac{w^2 \cdot \rk \cL}{4} \cdot \left(-  \mathfrak{X}_{B,\sigma}(r) \right) + O \left( \log r + \log T_{f, A}(r) \right).    \]
Since \(-\mathfrak{X}_{\sigma}(r)=O(\log r)\), it follows from this inequality that there exists an ample line bundle $A^\prime$ on $X$ such that the inequality $ T_{f, A^\prime}(r) \leqslant \log r $ holds for all $r \in \bR_{\geq 1}$ outside a Borel subset of finite Lebesgue measure. Since $ T_{f, A^\prime}(r)$ is a convex increasing function in $\log r$,  the function $r \mapsto T_{f, A^\prime}(r) / \log  r$ is monotone increasing, so that the inequality $ T_{f, A^\prime}(r) \leqslant \log r $ holds in fact for all $r \in \bR_{\geq 1}$ sufficiently big. We conclude by the criterion \cite[Proposition 5.9]{Griffiths-King} that $f$ is algebraic.
\end{proof}
\begin{proof}[Proof of Theorem \ref{thm:generalization of Nadel}]

First note the two following immediate results.
\begin{lemma}
Let $X$ be a complex algebraic variety and $\bar X, \bar X^\prime$ two compactifications of $X$. Then $\bar X$ is Borel hyperbolic modulo $\bar X - X$ if and only if $\bar X^\prime$ is Borel hyperbolic modulo $\bar X^\prime - X$.
\end{lemma}

\begin{lemma}
Let $\bar X$ be a complex algebraic variety and $\bar Y, \bar Z \subset \bar X$ two closed subvarieties. Then $\bar X$ is Borel hyperbolic modulo $\bar Z$ if and only if $\bar X$ is Borel hyperbolic modulo $\bar Z \cup \bar Y$ and $\bar Y$ is Borel hyperbolic modulo $\bar Y \cap \bar Z$.
\end{lemma}
 By applying the second lemma with $\bar X$ a compactification of $X$, $\bar Z := \bar X - X$ and $\bar Y $ a closed subvariety  that contains the singular locus of $\bar X$ and such that $\bar X \setminus \bar Y$ is affine, one sees that it is sufficient to consider the case where $X$ is smooth and affine. Fix a smooth projective compactification $\bar X$ of $X$ such that $D := \bar X - X$ is a normal crossing divisor. For every prime number $p$, let $\bar X(p)$ denote the normalization of $\bar X$ in the fraction field of $X(p)$ and $D(p) := \bar X(p) - X(p)$. Thanks to \cite[Proposition 2.4]{Bruni-level}, the augmented base locus of the canonical extension of the Griffiths parabolic line bundle $ \bar L_{\mathbb{V}}$ is contained in $D$. A fortiori, if one takes a sufficiently small $\epsilon >0$, the augmented base locus of $ \bar L_{\mathbb{V}}(- \epsilon \cdot D)$ is also contained in $D$. On the other hand, by \cite[Theorem 5.1]{Bruni-level}, the map $\pi_p : \bar X(p) \rightarrow \bar X$ ramifies over every irreducible component of $D$ with an order divisible by $p$ for almost all prime numbers $p$. Therefore, for almost all prime numbers $p$, the augmented base locus of the line bundle $(\pi_p^\ast \bar L_{\mathbb{V}})(- \epsilon \cdot p \cdot D(p)) = \bar L_{\pi_p ^\ast \mathbb{V}}(- \epsilon \cdot p \cdot D(p))$ is contained in $D(p)$, and we conclude using Theorem  \ref{thm:criterion for Borel hyperbolicity}.

\end{proof}

\section{Second Main Theorem for hyperbolically embedded complements}\label{sec:SMTHyp}
In this section, we give another application of the results of Section \ref{sec:Nevanlinna} by establishing a \emph{Second Main Theorem} for hyperbolically embedded complements. In Section \ref{ssec:Kobayashi}, we prove a \emph{Second Main Theorem} for the Kobayashi metric of the complement of a reduced divisor in a parabolic Riemann surface in case this complement is hyperbolic. Then, in Section \ref{ssec:SMTHyp}, we use this to establish a \emph{Second Main Theorem} for pairs \((X,D)\) such that \(X\setminus D\) is hyperbolically embedded in \(X\).  

\subsection{Applications to the Kobayashi metric}\label{ssec:Kobayashi}  Let \(B\) be a Riemann surface and let \(\Sigma\) be a reduced divisor on \(B\). If \(U:=B\setminus \Sigma\) is hyperbolic, then the universal cover of \(U\) is given by the unit disc \(\Delta\), and the \((1,1)\)-form  \(\omega_{\Delta}=\frac{i}{2}\frac{1}{(1-|z|^2)^2}dz\wedge d\bar{z}\) associated to the Poincaré metric on \(\Delta\) descends to the \((1,1)\)-form \(\omega_U\) associated to the Kobayashi metric on \(U\). The Gaussian curvature \(K\) of the Kobayashi metric on \(U\) is   constant and verifies \(K=-4\). We therefore have
\[\omega_U=-\frac{1}{4}\Ric \omega_U,\]
and the Schwarz lemma implies that \(\omega_U\) has Poincaré singularities around every point of \(\Sigma\). Moreover, one has
\begin{equation}[\omega_U]=\frac{1}{4}[-\Ric \omega_U]= \frac{\pi}{2}[\Sigma]-\frac{1}{4}\Ric[\omega_U].\label{eq:RicciKoba}\end{equation}
The fact that the left hand side is smaller than the right hand side is a consequence of the Schwarz lemma. To see that one has in fact an equality , one can use the precise behavior of the Kobayashi metric around the points of \(\Sigma\) (see for instance Section IV in \cite{FK92}). 

In case \(B\) is compact, integrating \eqref{eq:RicciKoba}, one obtains the following (see for instance Section IV in \cite{FK92})
\begin{equation}\int_B[\omega_U]= \frac{\pi}{2}(\deg \Sigma+2g(B)-2).\label{eq:GaussBonnet}\end{equation}
Lemma \ref{lem:RicciNegligeable} allows us to give the following Nevanlinna theoretic analogue of this result.

\begin{corollary}\label{cor:SMTKoba}
Let $B$ be a non-compact parabolic Riemann surface equipped with a parabolic exhaustion function and let \(\Sigma\) be a reduced divisor on \(B\) such that \(U:=B\setminus \Sigma\) is hyperbolic. Let \(\omega_U\) be the \((1,1)\)-form associated to the Kobayashi metric on \(U\). Then, \(\omega_U\) is locally integrable and one has
\[T_{[\omega_U]}(r)\leqslant_{{\rm exc}} \frac{\pi}{2}\big(N_{\Sigma}(r) -\mathfrak{X}_{\sigma}(r)\big) +O\left(\log r+ \log T_{[\omega_U]}(r)\right).\]
In particular, for any \(\ep >0\) one has 
\[(1-\ep)T_{[\omega_U]}\leqslant_{{\rm exc}} \frac{\pi}{2}(N_{\Sigma}(r)-\mathfrak{X}_{\sigma}(r)) + O\left(\log r \right).\]
\end{corollary}
\begin{proof}
For the first statement, one  integrates \eqref{eq:RicciKoba} and applies Lemma \ref{lem:RicciNegligeable} to obtain 
\begin{eqnarray*}T_{[\omega_U]}(r)&=&\frac{\pi}{2}N_{\Sigma}(r)-\frac{1}{4}\Ric[\omega_U]\leqslant_{\rm exc}\frac{\pi}{2}\big(N_{\Sigma}(r)- \mathfrak{X}_{\sigma}(r)\big)+O\left(\log r+ \log T_{[\omega_U]}(r)\right) \end{eqnarray*}
The second statement follows since for every \(\ep>0\),
\[\log T_{[\omega_U]}(r)\leqslant_{{\rm exc}}\ep T_{[\omega_U]}(r).\]
\end{proof}

\subsection{Second Main Theorem for hyperbolically embedded varieties}\label{ssec:SMTHyp}

Let \(X\) be a smooth complex projective variety. Recall that an open subset \(U\subset X\) is hyperbolically embedded in \(X\) if for every hermitian \(h\) metric on \(X\), there exists \(\eta>0\) such that \(\eta h\leqslant F_U\), where \(F_U\) denotes the Kobayashi-Royden infinitesimal pseudo-metric on \(U\). We refer to Chapter 3 of  \cite{Kob98} for the definition of the Kobayashi-Royden pseudo-metric and to Theorem 3.3.3 in \cite{Kob98} or Proposition 16 in \cite{PR07}  for the proof of the equivalent characterization we use here.

The algebraic counterpart of the \emph{Second Main Theorem} we shall prove below was established by Pacienza and Rousseau in \cite [Theorem 5]{PR07}. In order to emphasize the analogies between the analytic and the algebraic sides both in statements and proofs, we state and reprove  this result of Pacienza and Rousseau in the next theorem.

\begin{theorem}\label{thm:SMTHyp}
Let \(X\) be a smooth projective variety and let \(H\) be a reduced divisor on \(X\) such that \(X\setminus H\) is hyperbolically embedded in \(X\). Let \(A\) be  an ample line bundle on \(X\). Then
\begin{enumerate}
\item (Pacienza-Rousseau). There exists a constant \(\alpha_{\rm alg} >0\) such that for any projective curve \(B\) and any non-constant algebraic map \(f:B\to X\) such that \(f(B)\not\subset H\), one has
\[\deg f^*A\leqslant \alpha_{\rm alg}\big(\deg (f^*H)_{\rm red} - \chi(B)\big).\]
\item There exists a constant \(\alpha_{\rm an} >0\) such that for any non-compact parabolic Riemann surface $B$ equipped with a parabolic exhaustion function and any non-constant holomorphic map \(f:B\to X\) such that \(f(B)\not\subset H\), one has
\[T_{f,A}(r)\leqslant_{\rm exc} \alpha_{\rm an}\big(N^{(1)}_{f,H}(r) - \mathfrak{X}_\sigma(r)\big)+O(\log r).\]
\end{enumerate}
\end{theorem}

\begin{proof}
Let \(B\) be either a smooth projective curve or a parabolic Riemann surface with a non-constant holomorphic map \(f:B\to X\) such that \(f(B)\not\subset H\). Let \( \Sigma :=(f^*H)_{\rm red}\) be the set theoretical inverse image of the divisor \(H\). Let us set \(U=B\setminus \Sigma\) and \(V=X\setminus H\). By restriction, \(f\) induces a non-constant holomorphic map \({f}|_U:U\to V\). Since \(V\) is hyperbolic, so is \(U\). Let us denote by \(\omega_{U}\) the \((1,1)\)-form associated to the Kobayashi metric on \(U\) and let us denote by \(\|\cdot\|_U\) the associated norm.  Let us denote by \(F_V\) the Kobayashi Royden infinitesimal pseudo-norm on \(V\). The distance decreasing property of the Kobayashi-Royden pseudo-metric implies that 
\[f^*F_{V}\leqslant \|\cdot\|_{h_{U}}.\]

Let \(A\) be an ample line bundle on \(X\). Since  \(A\) is ample, it admits a hermitian metric with positive curvature \(\omega_A\). Since \(\omega_A\) is a  positive \((1,1)\)-form, it induces a hermitian metric \(h_A\) on \(X\). Since \(V\) is hyperbolically embedded in \(X\), there exists a positive real number \(\eta>0\) such that 
\[\eta \|\cdot\|_{h_A}|_{V}\leqslant F_{V}.\]
Therefore, one has
\[\eta f^*\|\cdot\|_{h_A}|_{U}\leqslant \|\cdot\|_{{U}}.\]
At the level of forms, this yields
\begin{equation}\eta f^*\omega_A|_{U}\leqslant \omega_{U}.\label{eq:FormulePreuve}\end{equation}
If \(B\) is projective, it suffices to integrate \eqref{eq:FormulePreuve} and apply  \eqref{eq:GaussBonnet}
\[\eta \deg f^*A=\eta \int_{B}f^*\omega_A=\eta \int_{U}f^*\omega_A\leqslant  \int_{U}\omega_{U}=\int_{B}[\omega_{U}]\leqslant  \frac{\pi}{2}(\deg \Sigma+2g(B)-2).\]
Taking \(\alpha_{\rm alg}=\frac{\pi}{2 \eta}\) yields the first assertion.

If \(B\) is a  non-compact parabolic Riemann surface equipped with a parabolic exhaustion function, it suffices to apply \(\int_{1}^r\int_{B(t)}\cdot \frac{dt}{t}\) to \eqref{eq:FormulePreuve} and apply  Corollary \eqref{cor:SMTKoba} to obtain that for any \(\ep>0\)
\[\eta T_{f,A}(r) \leqslant T_{[\omega_{U}]}(r)\leqslant_{\rm exc} \frac{\pi}{2(1-\ep)} \big(N^{(1)}_{f,H}(r)- \mathfrak{X}_{\sigma}(r)\big) + O(\log r).\]
It suffices to take \(\alpha_{\rm an}:=\frac{\pi}{2\eta(1-\ep)}\) to conclude the proof of the second assertion.
\end{proof}

\bibliography{biblio} 
 \bibliographystyle{alpha}

 \end{document}